\documentclass[a4paper,12pt]{article}
\usepackage{graphicx}
\usepackage[usenames,dvipsnames]{color}
\usepackage{epsfig}
\usepackage{t1enc}
\usepackage{inputenc}
\usepackage[noblocks, auth-lg]{authblk}

\newcommand{\CH}{{\mathcal H}}

\newcommand{\al}{\alpha}

\newcommand{\Sch}{Schr\"odinger}

\newcommand{\beq}{\begin{equation}}
\newcommand{\eeq}{\end{equation}}
\newcommand{\be}{\begin{equation}}
\newcommand{\ee}{\end{equation}}
\newcommand{\bea}{\begin{eqnarray}}
\newcommand{\eea}{\end{eqnarray}}
\newcommand{\brs}{\begin{eqnarray*}}
\newcommand{\ers}{\end{eqnarray*}}
\newcommand{\ba}{\begin{array}}
\newcommand{\ea}{\end{array}}
\newcommand{\br}{\begin{eqnarray}}
\newcommand{\er}{\end{eqnarray}}

\newcommand{\Lie}{\mathrm{Lie}}
\newcommand{\lp}{\left(}
\newcommand{\rp}{\right)}
\newcommand{\la}{\left\langle}
\newcommand{\ra}{\right\rangle}

\def\EOP{\ \hfill $\Box$}

\renewcommand{\r}[1]{(\ref{#1})}
\def\eps{\varepsilon}
\def\lb{\lambda}

\usepackage{amsthm,amssymb,amsbsy,amsmath,amsfonts,amssymb,amscd,mathrsfs}

\theoremstyle{plain}
\newtheorem{theorem}{Theorem}[section]
\newtheorem{corol}[theorem]{Corollary}
\newtheorem{lem}[theorem]{Lemma}
\newtheorem{lemma}[theorem]{Lemma}

\newtheorem{prop}[theorem]{Proposition}
\theoremstyle{definition}
\newtheorem{defn}[theorem]{Definition}

\theoremstyle{remark}
\newtheorem{rem}[theorem]{Remark}
\numberwithin{equation}{section}

\newcommand{\CV}{{\cal V}}

\newcommand{\N}{{\mathbf N}}
\newcommand{\C}{{\mathbf{C}}}
\newcommand{\R}{{\mathbf{R}}}
\newcommand{\Q}{{\mathbf{Q}}}
\newcommand{\Z}{{\mathbf Z}}

\newcommand{\g}{\gamma} 
\newcommand{\conv}{\operatorname{conv}}
\newcommand{\f}{\varphi}

\newcommand{\Id}{I_{\CH}}
\newcommand{\tr}{\operatorname{tr}}

\newcommand{\diag}{\operatorname{d}}

\newcommand{\ns}{{r}} 
\newcommand{\nss}{{\mathfrak{r}}} 
\newcommand{\ass}{\mathfrak{A}} 
\newcommand{\pro}{\Upsilon} 
\newcommand{\uni}{\hat\Upsilon} 
\newcommand{\K}{K} 

\begin{document}

\title{A weak spectral condition for the 
controllability of the
bilinear Schr\"odinger equation with application to the control of a rotating planar molecule}
\date{}

\author[a b c]{U.~Boscain} 
\author[d e]{M.~Caponigro} 
\author[e d]{T.~Chambrion} 
\author[c b]{M.~Sigalotti}

\affil[a]{Centre National de la Recherche Scientifique (CNRS)}
\affil[b]{CMAP, \'Ecole Polytechnique,  Route de Saclay, 91128 Palaiseau Cedex, France}
\affil[c]{INRIA, Centre de Recherche Saclay, Team GECO}
\affil[d]{INRIA, Centre de Recherche Nancy - Grand Est, Team CORIDA}
\affil[e]{Institut \'Elie Cartan, UMR 7502 Nancy-Universit\'e/CNRS, BP 239, Vand\oe uvre-l\`es-Nancy 54506,  France}

%
%

\maketitle

\begin{abstract}
In this paper we prove an approximate controllability result for the bilinear \Sch\ equation. This result requires less restrictive non-resonance hypotheses on the spectrum of the uncontrolled \Sch\ operator than those present in the literature. 
The control operator is not required to be bounded and we are able to extend the controllability result to the density matrices. 
The proof is based  on fine controllability properties of the finite dimensional Galerkin approximations and allows to get estimates for the $L^{1}$ norm of the control.
The  general controllability result is applied  to the problem of controlling  the rotation of a bipolar rigid molecule confined on a plane by means of two orthogonal external fields.
 \end{abstract}

\section{Introduction}

In this paper we are concerned  with the controllability problem for the 
Schr\"odinger equation
\begin{eqnarray}
i\frac{d\psi}{dt}=(H_0+u(t)H_1)\psi.
\label{eq-0}
\end{eqnarray}
Here $\psi$ belongs to the Hilbert sphere of a complex Hilbert space ${\cal H}$ and $H_0,H_1$ are self-adjoint operators on ${\cal H}$.
The control $u$ is scalar-valued and represents the action of an  external field.
The reference model is the one in which  $H_0=-\Delta+V(x)$, $H_1=W(x)$, where $x$ belongs to a domain $D\subset\R^n$ with suitable boundary conditions and $V,W$ are real-valued functions (identified with the corresponding multiplicative operators) characterizing  respectively the autonomous dynamics  and the coupling of the system with the control $u$. 
However, equation \r{eq-0} can be used to describe  more general controlled dynamics. For instance, a quantum particle on a Riemannian manifold subject to an external field (in this case $\Delta$ is the  Laplace--Beltrami operator) 
or a two-level ion trapped in a harmonic potential (the so-called Eberly and Law model~\cite{eberly-law,rangan, puel}).
In the last case, as in many others relevant physical situations, the operator $H_0$ cannot be written as the sum of a Laplacian plus a potential.  

Equation~\eqref{eq-0} is usually named {\it bilinear   Schr\"odinger equation} in the control community, 
the term {\it bilinear} referring to the linear dependence with respect to $\psi$ and the affine dependence with respect to $u$.  
(The term linear is reserved for systems of the form $\dot x=A x+ B u(t)$.) 
The operator
$H_0$ is usually called the \emph{drift}.

The controllability problem consists in establishing whether, for every pair of states $\psi_0$ and $\psi_1$, there exist a control $u(\cdot)$ and a time $T$ such that the solution of~\r{eq-0} with initial condition $\psi(0)=\psi_0$ satisfies $\psi(T)=\psi_1$. Unfortunately the answer to this problem is negative when $\CH$ is infinite dimensional. Indeed,  Ball, Marsden, and Slemrod proved in~\cite{bms} a result which implies (see~\cite{turinici}) that equation \r{eq-0} is not controllable in (the Hilbert sphere of) ${\cal H}$ 
Moreover, they proved that 
in the case in which  $H_0$ is the sum of the Laplacian and a potential in a domain $D$ of $\R^n$, equation~\r{eq-0}  is neither controllable in the Hilbert sphere $\mathbf{S}$ of $L^2(D,\C)$ nor in the natural functional space where the problem is formulated, namely the intersection of  $\mathbf{S}$ with the Sobolev spaces $H^2(D,\C)$ and $H^1_0(D,\C)$.
Hence one has to look for weaker controllability properties as, for instance, approximate controllability  or controllability  between 
the eigenstates of $H_0$ (which are the most relevant physical states).

However, in certain cases one can describe quite precisely the set of states that can be connected by admissible paths.
Indeed in \cite{beauchard,beauchard-coron} the authors prove that, in the case in which $H_0$ is the Laplacian on the interval $[-1,1]$, with Dirichlet boundary conditions, and $H_1$ is the operator of multiplication by $x$, 
the system is exactly controllable near the eigenstates in $H^7(D,\C)\cap \mathbf{S}$ (with suitable boundary conditions). 
%
This result was then refined in \cite{camillo}, where the authors proved that the exact controllability holds in $H^3(D,\C)\cap \mathbf{S}$, for a large class of control potentials (see also~\cite{fratelli-nersesyan}).

In dimension larger than one (for $H_0$ equal to the sum of the Laplacian and a potential) or for more general situations, the exact description of the reachable set appears to be more difficult and at the moment  only approximate controllability results are available. 

In~\cite{Schrod} an approximate controllability result for~\r{eq-0} was proved via finite dimensional geometric control techniques applied to the Galerkin approximations.  
The main hypothesis is that the spectrum of $H_0$ is discrete and without rational resonances, which means that the gaps between the eigenvalues of $H_0$ 
are $\Q$-linearly independent. Another crucial hypothesis appearing naturally is 
that the operator $H_1$ couples all eigenvectors of $H_0$.

The main advantages of that result with respect to those 
previously known
are that:
{\bf i)} it does not need $H_0$ to be of the form $-\Delta + V$; 
{\bf ii)}  it can be applied to the case in which $H_1$ is an unbounded operator; 
{\bf iii)} the control is a bounded function with arbitrarily small bound; 
{\bf iv)} it allows to prove controllability for density matrices and it can be generalized to prove approximate controllability results for a system of \Sch\ equations controlled by the same control (see~\cite{thomas-solo-paper}).

The biggest difficulty in order to apply the results given in~\cite{Schrod} to academic examples 
is that in most of the cases the spectrum is described as a simple numerical series (and hence $\Q$-linearly dependent). However, it has been proved that the hypotheses under which the approximate controllability results holds are generic~\cite{genericity-mario-paolo, corrige-mario-privat, genericity-mario-privat}. Notice that writing $H_0+u(t)H_1=(H_0+\eps H_1)+ (u(t)-\eps)H_1$ and redefining $(u(t)-\eps)$ as new control may be useful. As a matter of fact,  perturbation theory permits often to prove the $\Q$-linearly independence of the eigenvalues of $(H_0+\eps H_1)$ 
for 
most
values of $\eps$. This idea was used in \cite{Schrod} to prove the approximate controllability of the harmonic oscillator and the 3D potential  well with a Gaussian control potential.

Results similar to those presented in \cite{Schrod} have 
been obtained, with different techniques, in \cite{nersesyan} (see also \cite{ito-kunisch,beauchard-nersesyan,Nersy,fratelli-nersesyan}). 
They require less restrictive hypotheses on the spectrum of $H_0$ (which is still assumed to be discrete) but 
 they do not admit $H_1$ unbounded and do not apply to the density matrices. 
However, it should be noticed that~\cite{nersesyan} proves approximate controllability with respect to some Sobolev norm $H^s$, while the results given in~\cite{Schrod}
permit to get approximate controllability in the weaker norm $L^2$. 
As it happens for the results in~\cite{Schrod}, 
the  sufficient conditions for controllability obtained in \cite{nersesyan} are
generic.

Fewer controllability results are known in the case in which the spectrum of $H_0$ is not discrete. 
Let us mention the paper \cite{mirrahimi-continuous}, in which 
 approximate controllability is proved between wave functions corresponding to the discrete part of the spectrum (in the 1D case),
and \cite{ito-kunisch}.

In this paper we prove the approximate controllability of~\r{eq-0} under less restrictive hypotheses than those in~\cite{Schrod}.  More precisely, assume that $H_0$ has discrete spectrum $(\lambda_k)_{k\in\N}$ (possibly not simple)
and denote by $\phi_k$ an eigenvector of $H_0$ corresponding to $\lambda_k$ in such a way that $(\phi_{k})_{k\in \N}$ is an orthonormal basis of $\CH$. 
Let $\Xi$ be the subset of $\N^2$ given by all $(k_1,k_2)$ such that $\langle \phi_{k_1}, H_{1} \phi_{k_2}\rangle\ne 0$.
Assume that, for every $(j,k) \in \Xi$ such that $j\neq k$, we have $\lb_{j}\neq\lb_{k}$ (that is, degenerate energy levels  are not directly coupled by $H_{1}$).
We prove that the system is approximately controllable 
if there exists a subset $S$ of $\Xi$ such that 
the graph whose vertices are the elements of $\N$ and whose  edges are 
the elements of $S$ is connected (see Figure~\ref{fig:graphe})
 and, moreover, for 
every $(j_1,j_2)\in S$ and every $(k_1,k_2)\in \Xi$ different from $(j_1,j_2)$ and $(j_2,j_1)$, 
\begin{equation}\label{ugo-c}
|\lambda_{j_1}-\lambda_{j_2}|\ne |\lambda_{k_1}-\lambda_{k_2}|.
\end{equation}

As in \cite{Schrod}, $H_1$ is not required to be bounded and we are able to extend the controllability result to the density matrices and to simultaneous controllability (see Section~\ref{subsec:def} for precise definitions). This extension is interesting in the perspective of getting controllability results for open systems.

\begin{figure}
\begin{center}
\includegraphics[width=10cm]{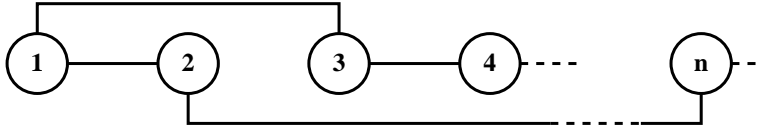}
\caption{Each vertex of the graph represents an eigenstate of $H_0$ (when the
spectrum is not simple, several nodes may be attached to the same
eigenvalue). An edge links two vertices if and only if $H_1$ connects
the corresponding eigenstates. In this example, $\langle \phi_1,
 H_1 \phi_2 \rangle $ and  $\langle\phi_1, H_1 \phi_3 \rangle $ are not
zero, while  $\langle \phi_1, H_1  \phi_4 \rangle =\langle \phi_2,
H_1  \phi_3 \rangle=\langle \phi_2, H_1  \phi_4 \rangle=0$.}
\label{fig:graphe}
\end{center}
\end{figure}

Interesting features of our result are that it permits to get $L^{1}$ estimates for the control laws
and that it does not require the spectrum of $A$ to be simple. Moreover,
beside requiring less restrictive hypotheses, this new result works better in academic examples where very often one has 
a spectrum which is resonant, but a lot of products $\langle \phi_{k_1}, H_{1} \phi_{k_2}\rangle$ which vanish. The consequence of the 
presence of these vanishing elements 
(i.e., of the smallness of $\Xi$)
is that less conditions 
of the type \eqref{ugo-c} need being verified. 

The condition on the spectrum given above is still generic and it is less restrictive than the one given in~\cite{nersesyan}, which corresponds to the case $S=\{(k_0,k)\mid k\in \N,\,k\ne k_0\}$ for some $k_0\in \N$ and where condition~\eqref{ugo-c} is required for every $(k_1,k_2)\in \N^2\setminus \{(j_1,j_2),(j_2,j_1)\}$.
Notice however that the approximate controllability result given in~\cite{nersesyan} is still in a stronger norm.

The idea of the proof is the following.
We recover approximate controllability for the system defined on an infinite dimensional Hilbert space through 
fine controllability properties of the $N$-dimensional Galerkin approximations, $N \in \mathbf{N}$, which allow us to pass to the limit as $N\to\infty$.
More precisely, we prove that, for $n,N\in \N$ with $N\gg n\gg 1$, for given initial and final conditions $\psi_0$, $\psi_1$ in the Hilbert sphere of ${\cal H}$ which are linear combinations of the first $n$ eigenvectors of $H_0$,  
it is possible to steer $\psi_0$ to $\psi_1$ in the Galerkin approximation of order $N$ in such a way that the projection on the components 
$n+1,\ldots,N$ has arbitrarily small norm along the trajectory. 
This kind of controllability for the Galerkin approximation of order $N$ is proved in two steps: firstly, thanks to a time-dependent change of variables we transform the system in a driftless one,
nonlinear  in the control, and we prove the result up to phases. The  
change of variables was already introduced in~\cite{agrachev-chambrion,Schrod}; the technical novelty of this paper 
is the convexification analysis for the transformed system, which allows to conclude the controllability with less restrictive non-resonance hypotheses.   
Secondly, 
the control of phases is obtained via a classical method, using as pivot an eigenstate of $H_0$ and exploiting the controllability (up to phases) of the time-reversed  Schr\"odinger equation. 
This last step 
requires 
some further arguments
in the case of 
simultaneous controllability.

\begin{figure}
\begin{center}
\includegraphics[width=9.3cm]{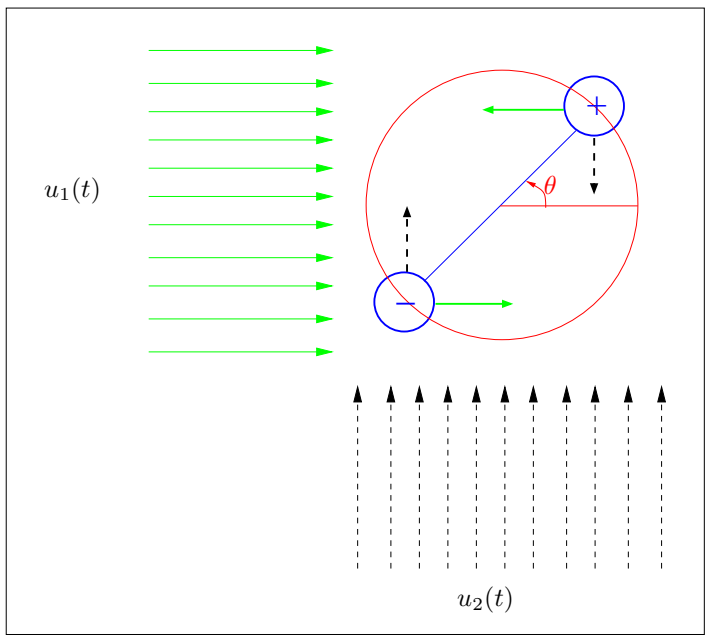}
\caption{The bipolar rigid molecule confined to a plane.}
\label{fig:molecola}
\end{center}
\end{figure}

In the second part of the paper we apply our result to the problem of controlling a bipolar rigid molecule confined on a plane by means of two electric fields constant in space and controlled in time, oriented along two orthogonal directions (see Figure~\ref{fig:molecola}).
The corresponding Schr\"odinger equation can be written as 
 \begin{equation}\label{eq-1}
i\frac{\partial\psi(\theta,t)}{\partial t}=\left(  -\frac{\partial^2}{\partial \theta^2}+u_1(t) \cos(\theta)+u_2(t) \sin(\theta) \right)\psi(\theta,t),\quad\theta\in \mathbb{S}^{1},
\end{equation}
where $\mathbb{S}^{1} = \R/2\pi\Z$.
 Notice that a controlled rotating molecule 
is the most relevant physical application 
for which  the spectrum of $H_0$ is discrete.

The system described by~\eqref{eq-1} is not controllable if we fix one control to zero. 
Indeed by parity reasons
the potential $\cos(\theta)$ does not couple an odd wave function with an even one and the potential
 $\sin(\theta)$ does not
 couple  wave functions with the same parity.

Up to our knowledge, 
the controllability result presented in this paper is the only one which can be directly applied to  system~\r{eq-1}. 
Indeed, both the  results obtained in  \cite{Schrod} and \cite{nersesyan} seem to require
sophisticated perturbation arguments in order to conclude the approximate controllability of \r{eq-1}.

The proof of the controllability of \r{eq-1} by means of the controllability result obtained in this paper is not immediate since we have to use in a suitable way the two controls. The idea is to prove that, given an initial condition $\psi_0$ which is even with respect to some $\bar \theta\in \mathbb{S}^{1}$, by varying the control $u=(u_1,u_2)\in\R^2$ along the line $\R(\cos(\bar\theta),\sin(\bar\theta))$, 
it is possible to steer it (approximately) towards any other wave function even with respect to $\bar \theta$.
In particular, it is possible to steer any eigenfunction (which is necessarily even with respect to some $\bar\theta$) to the ground state (that is, the constant function $1/\sqrt{2\pi}$), which can, in turn, be steered towards any other eigenfunction. The argument can be refined to prove approximate controllability among any pair of wave functions on the Hilbert sphere.

The structure of the paper is the following. In Section~\ref{sec:Framework} we introduce the class of systems under consideration and we discuss their well-posedness. Then, we state the main results contained in the paper. Section~\ref{sec:finitedimensional} is devoted to the case in which $\CH$ is finite dimensional. Sections~\ref{sec:conv}, \ref{sec:modulus}, and~\ref{sec:phase} contain the proof of the main results and in Section~\ref{sec:L1} we present estimates on the $L^{1}$ norm of the control. Section~\ref{sec:Example} contains an application to the infinite potential well, showing controllability and establishing $L^{1}$ estimates of the control. 
Section~\ref{sec:molecule} provides the application to the bipolar planar molecule evolving on the plane.


\section{Framework and main results}\label{sec:Framework}

\subsection{Settings and notations}

As in \cite{Schrod}, we 
 use an abstract framework instead of 
a presentation in terms of partial differential equations.
The advantage of this presentation is that it is very versatile
and applies without modification for Schr\"odinger equation on a (possibly unbounded) domain of $\mathbf{R}^n$
or on a manifold such as $
\mathbb{S}^1$ (see Section~\ref{sec:molecule}). 
{To avoid confusion, let us stress that $H_{0}$ and $H_{1}$, introduced in the introduction, are self-adjoint operators while $A = -i H_{0}$ and $B = -i H_{1}$, used in what follows, are skew-adjoint.}
%
Hereafter $\N$ denotes the set of strictly positive integers.
We also denote by $\mathbf{U}(\CH)$ the space of unitary operators on $\CH$.

\begin{defn}\label{DEF_skew_adjoint_bilinear_diagonal}
Let $\cal H$ be an Hilbert space with scalar product $\langle \cdot ,\cdot \rangle$ and  $A,B$ be two (possibly unbounded) linear operators on $\cal H$,
with domains $D(A)$ and $D(B)$. Let $U$ be a subset of $\mathbf{R}$.
Let us introduce the formal controlled equation
\begin{equation} \label{eq:main}
\frac{d\psi}{dt}(t)=(A+u(t)B) \psi(t), \quad u(t) \in U.
\end{equation}
We say that $(A,B,U,\Phi)$ satisfies $(\mathfrak{A})$ 
if the following assumptions are verified:
\begin{itemize}
\item[ ] 
\begin{description}
\item[($\mathfrak{A}1$)] $ \Phi = (\phi_k)_{k \in \mathbf{N}}$ is an Hilbert basis of $\cal H$
made of eigenvectors of $A$ associated with the family of eigenvalues $(i \lb_{k})_{k \in \mathbf{N}}$; 
\item[($\mathfrak{A}2$)] $\phi_k \in D(B)$ for every $k \in
\mathbf{N}$;
\item[($\mathfrak{A}3$)] 
$A+uB:\mathrm{span}\{\phi_k \mid k\in\N\}\to \CH$ is essentially skew-adjoint for every $u\in U$;
\item[($\mathfrak{A}4$)] if $j\neq k$ and $\lb_{j} = \lb_{k}$ then $\la \phi_{j},B\phi_{k} \ra = 0$.
\end{description}
\end{itemize}
\end{defn}

\begin{rem}
If $A$ has simple spectrum then $(\mathfrak{A}4)$ is verified. 
If all the eigenvalues of $A$ have finite multiplicity, then, up to a change of basis,
hypothesis $(\mathfrak{A}4)$ is a consequence of $(\mathfrak{A}1-2-3)$. 
\end{rem}

A crucial consequence of assumption $(\ass3)$ is that, for every constant $u$ in $U$,
$A+uB$
generates a group
of unitary transformations $e^{t(A+uB)}:{\cal H}\rightarrow {\cal H}$. The unit sphere
of $\cal H$ is invariant for all these transformations.

\begin{defn}\label{DEF_Sol_skew_adjoint_bilinear_system}
Let $(A,B,U,\Phi)$ satisfy $(\mathfrak{A})$ and
$u:[0,T]\rightarrow U$
be piecewise constant.
The \emph{solution} of (\ref{eq:main}) with initial condition $\psi_0 \in {\cal H}$ is
\begin{equation}\label{EQ_def_solution}
\psi(t)=\pro^{u}_{t} (\psi_{0}),
\end{equation}
where $\pro^{u}:[0,T]\rightarrow \mathbf{U}({\cal H})$ is  the \emph{propagator} of~\eqref{eq:main} that associates, with every $t$ in $[0,T]$,  the unitary linear transformation 
 $$
\pro^{u}_{t} =    e^{(t-\sum_{l=1}^{j-1} t_l)(A +u_j B)}\circ e^{t_{j-1}(A +u_{j-1} B)}\circ \cdots \circ e^{t_1(A+u_1  B)},
$$
where $\sum_{l=1}^{j-1} t_l\leq t<\sum_{l=1}^{j} t_l$ and
$u(\tau)=u_j$ if $\sum_{l=1}^{j-1} t_l\leq \tau<\sum_{l=1}^{j} t_l$.
\end{defn}

The notion of solution introduced above makes sense in very degenerate situations and can be enhanced when $B$ is bounded (see~\cite{bms} and references therein).

Note that, since
$$
\la \phi_{n}, e^{t(A+uB)} \psi_{0}\ra = \la  e^{-  t(A+uB)} \phi_{n}, \psi_{0} \ra\,,
$$
for every $n\in \N$,
$\psi_{0} \in \CH$, and $u \in U$, 
then,
for every solution $\psi(\cdot)$
 of~\eqref{eq:main}, the function $t \mapsto \la \psi(t),\phi_{n} \ra $ is absolutely continuous and satisfies, for almost every $t\in [0,T]$,
\begin{equation}\label{EQ_very_weak}
\frac d{d t}\la \phi_n, \psi(t)\ra=-\la (A+u(t)B)\phi_n, \psi(t)\ra.
\end{equation}

\subsection{Main results}\label{subsec:def}

As already recalled in the introduction, exact controllability 
is hopeless in general. 
Several relevant definitions of approximate controllability are available. 
The first one is the standard approximate controllability.

\begin{defn}\label{DEF_pointwise_controllabilty}
Let $(A,B,U,\Phi)$ satisfy $(\ass)$.
We say that \eqref{eq:main} is \emph{approximately controllable}
if for every $\psi_0,\psi_1$ in the unit sphere of $\cal H$
 and every $\eps>0$ there exist a piecewise constant control function $u:[0,T] \to U$ such that 
$
\|\psi_1-  \pro^{u}_{T}(\psi_{0})\| <\eps.
$
\end{defn}

Recall that $A$ has purely imaginary eigenvalues $(i\lb_{k})_{k\in \mathbf{N}}$ with associated eigenfunctions $(\phi_{k})_{k\in \mathbf{N}}$.
Next we introduce the notion of \emph{connectedness chain}, whose existence is crucial for our result. 


\begin{defn}
 Let $(A,B,U,\Phi)$ satisfy $(\mathfrak{A})$. A subset $S$ of
$\mathbf{N}^2$ \emph{couples} two levels $j,k$ in $\mathbf{N}$,
if
there exists a finite sequence $\big ((s^{1}_{1},s^{1}_{2}),\ldots,(s^{p}_{1},s^{p}_{2}) \big )$
in $S$ such that
\begin{description}
\item[$(i)$] $s^{1}_{1}=j$ and $s^{p}_{2}=k$;
\item[$(ii)$] $s^{j}_{2}=s^{j+1}_{1}$ for every $1 \leq j \leq p-1$;
\item[$(iii)$] $\langle  \phi_{s^{j}_{1}}, B \phi_{s^{j}_{2}}\rangle \neq 0$ for $1\leq j \leq
p$.
\end{description}

$S$ is called a \emph{connectedness chain} (respectively $m$-connectedness chain)  for $(A,B,U,\Phi)$ if $S$ (respectively $S \cap \{1,\ldots,m\}^{2}$) couples every pair of levels in $\mathbf{N}$ (respectively in $\{1,\ldots,m\}$).

A connectedness chain is said to be \emph{non-resonant} if for
every $(s_1,s_2)$ in $S$, $|\lb_{s_1}-\lb_{s_2}|\neq |\lb_{t_1}-\lb_{t_2}|$ for every $(t_{1},t_{2})$ in
$\mathbf{N}^2\setminus\{(s_1,s_2),(s_2,s_1)\}$ such that $\langle \phi_{t_{2}}, B \phi_{t_{1}}\rangle  \neq 0$.
\end{defn}


\begin{theorem}\label{THE_Control_collectively}
 Let $\delta >0$ and let
  $(A,B,[0,\delta],\Phi)$ satisfy $(\mathfrak{A})$. 
 If there exists a non-resonant connectedness chain for $(A,B,[0,\delta],\Phi)$
then
\eqref{eq:main} is approximately controllable.
\end{theorem}

Theorem~\ref{THE_Control_collectively} is a particular case of Theorem~\ref{THE_Control_wave_function}, stated in the next section.


\begin{rem}
 Notice that in the assumptions of
 Theorem~\ref{THE_Control_collectively} 
we do not require that the eigenvalues of $A$ are simple.
Take for instance ${\cal H}=\C^4$, $U=[0,1]$, and 
$$
A=\left(\begin{array}{cccc}
i& 0& 0& 0\\
 0& 2i& 0& 0\\
 0& 0& 4i& 0\\
  0&0& 0& 4i\end{array}\right),
  \qquad
B=\left(\begin{array}{cccc}
0& 1& 1& 0\\
 -1& 0& 0& 1\\
 -1& 0& 0& 0\\
  0&-1& 0& 0\end{array}\right).
$$
A connectedness  chain  is given by $\{(1,2),(2,1),(1,3),(3,1),(2,4),(4,2)\}$. The corresponding eigenvalue gaps are $|\lambda_2-\lambda_1|=1$, $|\lambda_3-\lambda_1|=3$, and $|\lambda_4-\lambda_2|=2$. Hence, the connectedness chain 
is non-resonant. 

%

\end{rem}

The following proposition gives an estimate of the $L^{1}$ norm of the control steering~\eqref{eq:main} from one eigenvector to an $\eps$-neighborhood of another.
A generalization of this proposition is given by Theorem~\ref{PRO_L1_estimates}.

\begin{prop}\label{prop:L1estimatessimple}
Let $\delta >0$. Let
  $(A,B,[0,\delta],\Phi)$ satisfy $(\mathfrak{A})$
   and admit  a non-resonant chain of connectedness $S$.
Then for every  $\eps>0$ and $(j,k)\in S$ there exist a piecewise constant control $u:[0,T_u]\rightarrow [0,\delta]$ and $\theta \in \R$ such that 
$\| \pro^u_{T_u}(\phi_j)-e^{i \theta} \phi_{k}\|<\eps$ and
$$
\|u\|_{L^1}\leq \frac{5 \pi}{4  |\la \phi_{k}, B \phi_{j} \ra |}\,.
$$
\end{prop}

\subsection{Simultaneous controllability and controllability in the sense of density matrices}

We define now a notion of controllability in the sense of density matrices.
Recall that 
 a density matrix $\rho$ 
 is a non-negative, self-adjoint operator of trace class whose trace is normalized to one.
Its time evolution 
is determined by 
$$
\rho(t)=\pro^{u}_{t} \rho(0) \pro^{u^{\ast}}_{t}
$$
where $ \pro^{u^{\ast}}_{t}$ is the adjoint of $ \pro^{u}_{t}$.
Notice that the spectrum of $\rho(t)$ is constant along the motion, since, for every $t$, $\rho(t)$ is unitarily equivalent to $\rho(0)$.

\begin{defn}\label{DEF_density_matrices_controllabilty}
 Let $(A,B,U,\Phi)$ satisfy $(\ass)$.
We say that \eqref{eq:main} is \emph{approximately controllable in the sense of the density matrices} 
if for every pair of unitarily equivalent density matrices $\rho_0,\rho_1$ 
 and every $\eps>0$ there exists a piecewise constant control $u:[0,T]\rightarrow U$ such that
 $$\left \|  \rho_{1} - \pro^{u}_T \rho_0 \pro^{u*}_{T} \right \|<\eps,$$
in the sense of the operator norm induced by the Hilbert norm of $\cal H$.
\end{defn}

\begin{defn} \label{DEF_simultaneous_controllability}
 Let $(A,B,U,\Phi)$ satisfy $(\ass)$. 
We say that \eqref{eq:main} is \emph{approximately simultaneously controllable} if for every $\ns$ in $\N$, $\psi_1,\ldots,\psi_{\ns}$ in $\CH$, $\uni$ in $\mathbf{U}(\CH)$, and $\eps>0$  there exists a piecewise constant control $u:[0,T]\rightarrow U$
 such that, for every $1\leq k \leq \ns$,
$$
\left \| \uni \psi_k - \pro^{u}_T \psi_k \right \|<\eps.
$$
\end{defn}

The following result is proved in Sections~\ref{sec:conv},~\ref{sec:modulus}, and~\ref{sec:phase}.

\begin{theorem}\label{THE_Control_wave_function}
Let $\delta >0$ and let
  $(A,B,[0,\delta],\Phi)$ satisfy $(\mathfrak{A})$. 
If there exists a non-resonant connectedness chain for $(A,B,[0,\delta],\Phi)$, then
\eqref{eq:main} is approximately simultaneously controllable.
\end{theorem}

Simultaneous controllability implies controllability in the sense of density matrices (see
Proposition~\ref{prop:relations}). Hence we have the following.

\begin{corol}\label{THE_Control_density_matrices}
Let $\delta >0$ and let
  $(A,B,[0,\delta],\Phi)$ satisfy $(\mathfrak{A})$. 
 If there exists a non-resonant connectedness chain for $(A,B,[0,\delta],\Phi)$, then
\eqref{eq:main} is approximately controllable in the sense of the density matrices.
\end{corol}

\begin{theorem}\label{PRO_L1_estimates}
Let $(A,B,U,\Phi)$ satisfy $(\mathfrak{A})$ and admit  a non-resonant chain of connectedness.
Then there exists a basis $\hat\Phi=(\hat\phi_{k})_{k\in\N}$ of eigenvectors of $A$ and a 
subset $S$ of $\N^2$
 such that,
for every $m \in \mathbf{N}$,
$S$
is a $m$-connectedness chain for $(A,B,U,\hat \Phi)$.
Moreover, let $\delta > 0$ and $U=[0,\delta]$, then for every  $\eps>0$ and for every permutation $\sigma:\{1,\ldots,m\} \to \{1,\ldots,m\}$  there exist a piecewise constant control $u:[0,T_u]\rightarrow [0,\delta]$ and $\theta_1,\ldots ,\theta_m$ in $\R$ for which the propagator $\pro^u$ of~\eqref{eq:main} satisfies
$\|\pro^u_{T_u}\hat \phi_l-e^{i \theta_l} \hat \phi_{\sigma(l)}\|<\eps$ for every $1\leq l\leq m$ and
$$
\|u\|_{L^1}\leq \frac{5~\pi~ (2^{m-1}-1)}{4 \inf \{|\la \hat\phi_{k} , B \hat\phi_{j}\ra | \,:\,  (j,k)\in S,\ 1\leq j,k\leq m\}}.
$$
\end{theorem}

The proof of the first part of the statement of Theorem~\ref{PRO_L1_estimates} is given in Section~\ref{sec:nconn}. The second part is proved in Section~\ref{sec:L1}.

Notice that a lower bound on the $L^{1}$ norm of the control was already proved in~\cite{Schrod} (see Proposition~\ref{PRO_L1_lower_estimates}).

\section{Finite dimensional case}
\label{sec:finitedimensional}

Denote by $\mathfrak{u}(n)$ and $\mathfrak{su}(n)$ the Lie algebras of the group of unitary matrices $U(n)$ and its special subgroup $SU(n)=\{M  \in U(n)|\det M=1\}$ respectively.

Here we address the case where $\CH$ is of finite dimension $n$.  Equation~\eqref{eq:main} then defines a bilinear control system on $U(n)$.
Finite dimensional systems of the type~\eqref{eq:main} have been extensively  studied.  A necessary and sufficient condition 
for controllability on $SU(n)$ (i.e. the property that every two points of $SU(n)$ can be joined by a trajectory in $U(n)$ of system~\eqref{eq:main}) is that the Lie algebra generated by $A$ and $B$ contains $\mathfrak{su}(n)$. 
This criterion is optimal, yet  sometimes too complicated to be checked for $n$ large.   
Easily verifiable sufficient conditions for controllability on $SU(n)$ have been thoroughly studied in the literature (see for instance~\cite{dalessandro-book} and references therein).
Next proposition gives a new sufficient condition,  slightly  improving those in~\cite{turinici} and~\cite[Proposition~4.1]{Schrod}. Its proof is based on the techniques that we extend to the infinite dimensional case in the following sections. 

The controllability result is obtained under a slightly weaker assumption than $(\ass)$.

\begin{prop}\label{PRO_contr_dim_finie}
{Let $\CH=\C^n$.}
Let $(A,B,U,\Phi)$ satisfy $(\mathfrak{A}1-2-3)$ and admit a non-resonant
connectedness chain $S$. 
Assume, moreover, that  $\lb_{j} \neq \lb_{k}$ for every $(j,k) \in S$.
Then the control system~\eqref{eq:main} is controllable  both on the unit sphere of $\mathbf{C}^n$
and on $SU(n)$, provided that $U$ contains at least two points. 
If, moreover, $\tr A \neq 0$ or $\tr B \neq
0$, then the control system \eqref{eq:main} is controllable on $U(n)$.
\end{prop}

\begin{proof}
For  every $1\leq j,k \leq n$, let $e^{(n)}_{jk}$ be the $n\times n$ matrix whose
entries are all zero, but the one at line $j$ and column $k$ which is equal to $1$.
We denote by $a_{jk}$ and $b_{jk}$ the $(j,k)$-th entry of $A$ and $B$, respectively.

Recall that, for any two
$n\times n$ matrices $X$ and $Y$, $\mbox{ad}_X(Y)=[X,Y]=XY-YX$, and compute the
iterated matrix commutator
$$
\mbox{ad}^p_{A}(B)=\sum_{j,k=1}^n (a_{jj}-a_{kk})^p b_{jk}
e^{(n)}_{jk}.
$$

Fix $(j,k)$ in $S$. 
By hypothesis, for every $l,m$ in
$\{1,\ldots, n\}$ such that $\{l,m\}\neq \{j,k\}$, $(a_{jj}-a_{kk})^2\neq
(a_{ll}-a_{mm})^2$ or $b_{lm}=0$.  There exists some polynomial $P_{jk}$ with
real coefficients such that $P_{jk}((a_{jj}-a_{kk})^2)=1$ and
$P_{jk}((a_{ll}-a_{mm})^2)=0$ if $(a_{jj}-a_{kk})^2\neq (a_{ll}-a_{mm})^2$.
Let
$P_{jk}=\sum_{h=0}^d c_h X^h$.
Then
$$
\sum_{h=0}^d c_h\mbox{ad}^{2h}_{A}(B)=
b_{jk} e^{(n)}_{jk} +b_{kj} e^{(n)}_{kj}.
$$
As a consequence,
$$
\sum_{h=0}^d c_h\mbox{ad}^{2h+1}_{A}(B)=(a_{jj}-a_{kk})\left(b_{jk} e^{(n)}_{jk} + \overline{b_{jk}} e^{(n)}_{kj}\right) = i(\lb_{j} - \lb_{k})\left(b_{jk} e^{(n)}_{jk} + \overline{b_{jk}} e^{(n)}_{kj}\right),
$$
and then the two elementary Hermitian matrices $e^{(n)}_{jk}-e^{(n)}_{kj}$ and
$ie^{(n)}_{jk}+ie^{(n)}_{kj}$
also belong to $\mbox{Lie}(A,B)$.
Because of the connectedness of $B$ and thanks to the relation 
\begin{align*}
\left[e^{(n)}_{jk},e^{(n)}_{lm}\right] &= \delta_{kl} e^{(n)}_{jm}   -  \delta_{jm}e^{(n)}_{lk},
\end{align*}
one deduces that
 $\mathfrak{su}(n) \subset \mbox{Lie}(A,B)$.

If $\tr A=\tr B=0$, then $A$ and $B$
belong to $\mathfrak{su}(n)$, hence $\mathfrak{su}(n) =
\mbox{Lie}(A,B)$.
If $\tr A\neq 0$ or $\tr B\neq 0$, then $A$ or
$B$
does not belong to $\mathfrak{su}(n)$ and $\mathfrak{u}(n) =
\mbox{Lie}(A,B)$.
This completes the proof of the controllability of the control system
(\ref{eq:main})
on $SU(n)$ and $U(n)$.

It remains to prove the controllability on the unit sphere $\mathbf{S}^{n}$ of
$\mathbf{C}^n$.
Fix $x_0, x_{1}$ in $\mathbf{S}^{n}$, and consider
an element  of  $g_{1} \in SU(n)$ such that $g_{1} x_{0} = x_{1}$.
According to what precedes there exists a trajectory $g$ in $U(n)$ of~\eqref{eq:main} from $I_{n}$ to $g_{1}$.
The curve $t \mapsto g(t) x_{0}$ is a trajectory of~\eqref{eq:main} in $\mathbf{S}^{n}$ that links $x_{0}$ to $x_{1}$.
\end{proof}


\section{Convexification procedure}
\label{sec:conv}

Sections~\ref{sec:conv}, \ref{sec:modulus}, and \ref{sec:phase} are devoted to the proof of Theorem~\ref{THE_Control_wave_function} in the case in which $\CH$ has infinite dimension.

\subsection{Time-reparametrization}\label{SEC_time_reparametrization}

We denote by $PC$ the set of
piecewise constant functions $u: [0,\infty) \to [0,\infty)$ such that there exist  $u_{1},\ldots, u_{p} > 0$ and 
$0 = t_{1} < \cdots < t_{p+1} = T_{u}$ for which
$$
u:t\mapsto \sum_{j=1}^p u_j \chi_{[t_j,t_{j+1})}(t).
$$
Let us identify $u = \sum_{j=1}^p u_j \chi_{[t_j,t_{j+1})}$ with the finite sequence $(u_j,\tau_j)_{1\leq j \leq p}$
where $\tau_j= t_{j+1}-t_j$ for every $1\leq j \leq p$.

  We define the map
$$\begin{array}{llcl}
   {\cal P} :& PC & \rightarrow & PC \\
	& (u_j,\tau_j)_{1\leq j \leq p} & \mapsto & \left ( \frac{1}{u_j}, u_j
\tau_j \right ),
  \end{array}$$
which satisfies the following easily verifiable properties.
\begin{prop}\label{prop:4321}
For every $u\in PC$,
${\cal P}\circ {\cal P} (u)= u$
and
$\|{\cal P}(u)\|_{L^1}=\sum_{i=1}^p \tau_j.$
\end{prop}

Assume that $(A,B,U, \Phi)$ satisfies $(\ass)$.
In analogy with Definition~\ref{DEF_Sol_skew_adjoint_bilinear_system}, we define,
for every $u= \sum_{j=1}^p u_j \chi_{[t_j,t_{j+1})} \in PC$ such that $u(t) \in U$ for every $t\geq 0$, the solution of
\begin{equation}\label{EQ_main_reparam}
\frac{d\psi}{dt}(t)=(u(t)A+B)\psi(t),
\end{equation}
 with initial condition $\psi_{0} \in \CH$ as
$$
\psi(t)= e^{(t-t_{l})(u_l A +B)} \circ \cdots \circ e^{t_1(u_1 A +B)}(\psi_0)\,,
$$
where $t_{l} \leq t \leq t_{l+1}$.

System~\eqref{EQ_main_reparam} is the time reparametrization of
system~(\ref{eq:main}) induced by the transformation $\cal P$, as stated in the following proposition.

\begin{prop}\label{PRO_reparam}
 Let $u=(u_j,\tau_j)_{1\leq j \leq p}$ belong to $PC$ and $\psi_0$ be a point of $\cal H$.
Let $\psi$ be the solution of (\ref{eq:main}) with control $u$ and initial
condition $\psi_0$, and
$\widetilde{\psi}$ be the solution of (\ref{EQ_main_reparam}) with control ${\cal P}(u)$ and
initial condition $\psi_0$.
Then 
$
\psi\left (T_{u} \right )= \widetilde{\psi}\left (\|u\|_{L^1}
\right )
$.
\end{prop}
\begin{proof}
 It is enough to remark that, if $u\neq 0$, for every $t\in [0,\infty)$,
$e^{t(A+uB)}=e^{tu \left ( \frac{1}{u} A +B \right )}.$
\end{proof}
As a consequence of Proposition~\ref{PRO_reparam} it is equivalent to prove controllability for~\eqref{eq:main} with $U = (0,\delta]$ or
to prove  controllability for system~(\ref{EQ_main_reparam}) with control  $u \in [1/\delta,\infty)$.

\subsection{Convexification}

For every positive integer $N$  let the matrices  
$$
A^{(N)} = \mathrm{diag}(i \lb_{1},\ldots, i\lb_{N})\quad \mbox{ and }\quad
B^{(N)} = (\langle \phi_{j}, B\phi_{k} \rangle)_{j,k =1}^{N} =: (b_{jk})_{j,k =1}^{N} \,,
$$
be the Galerkin approximations at order $N$ of $A$ and $B$, respectively.
Let $t\mapsto \psi(t)$ be a solution of
$$
\dot{\psi} = (u A^{(N)} + B^{(N)}) \psi\,,
$$
corresponding to a control function $u$ and consider $v(t) = \int_{0}^{t} u(\tau) d\tau$. Denote by  $\diag(B)$ the diagonal of $B^{(N)}$ and let $\hat B^{(N)} = B^{(N)} - \diag(B)$. 
Then $q: t\mapsto e^{-v(t)A^{(N)} - t \diag(B)}\psi(t)$,
is a solution of
$$
\dot q(t)=e^{-v(t)A^{(N)}- t \diag(B)} \hat B^{(N)} e^{v(t)A^{(N)} + t \diag(B)}q(t). \eqno{(\Theta_N)}
$$
Let us set 
\begin{equation}\label{eq:vartheta}
\vartheta_{N}(t,v) = e^{-vA^{(N)}- t \diag(B)} \hat B^{(N)} e^{vA^{(N)} + t \diag(B)}.
\end{equation}

\begin{lem}\label{LEM_convexification}
Let $\K$ be a positive integer and $\g_{1}, \ldots, \g_{\K} \in \R\setminus \{0\}$ be  such that
$|\g_{1}|\neq|\g_{j}|$ for $j =2,\ldots, \K.$
Let 
$$ 
\f(t) = (e^{it\g_{1}}, \ldots, e^{it\g_{\K}}).
$$
Then, for every $t_0\in \R$, we have
$$ 
\overline{\conv{\f([t_0,\infty))}} \supseteq \nu \mathbb{S}^{1} \times \{(0, \ldots, 0)\}\,,
$$
where
$
\nu=\prod_{k=2}^{\infty} \cos \left (\frac{\pi}{2 k} \right ) >0.
$
Moreover,
for every $R>0$ and $\xi \in \mathbb{S}^{1}$ there exists a sequence $(t_{k})_{k\in \N}$ such that
$
t_{k+1} - t_{k} > R
$
and 
$$
\lim_{h\to \infty } \frac{1}{h} \sum_{k=1}^{h} \varphi (t_{k}) = (\nu \xi,0,\ldots,0)\,.
$$
\end{lem}

\begin{proof}
Since 
\begin{equation}\label{eq:rotinv}
\f(t-t_0) = (e^{-it_0\g_{1}}e^{it\g_{1}}, \ldots, e^{-it_0\g_{\K}}e^{it\g_{\K}}),
\end{equation}
it is enough to prove the lemma for $t_0=0$.
We can suppose that $|\g_{1}| = 1$ and,
up to a reordering of the indexes, that there exist $n$ and $\tilde{n}$  such that $1\leq n\leq \tilde{n} \leq \K$, 
$|\g_{i}|\neq|\g_{j}|$ for every $i,j\in\{1,\dots, n\}$,   
$\g_{2},  \ldots, \g_{\tilde{n}} \in \Z$, $\g_{\tilde{n}+1},\ldots, \g_{\K} \in \R\setminus \Z$, and
$\{|\g_{n+1}|,\ldots,|\g_{\tilde{n}}|\} \subset \{|\g_{2}|,\ldots,|\g_{n}|\}$.

Consider the $2^{n-1}$ real numbers defined as follows:
let
$$
\bar{t}_1 = 0,
$$
and for $k \in \{1,\ldots,n-1\}$ and $j\in\{1,\dots,2^{k-1}\}$,
\begin{equation*}
\bar{t}_{2^{k-1}+j} = \bar{t}_{j} +\frac{\pi}{|\g_{k+1}|}\,.
\end{equation*}
Up to a reordering of the $\bar{t}_{j}$, we can suppose that $0=\bar{t}_{1}< \bar{t}_{2}< \cdots< \bar{t}_{2^{n-1}}$. Take  an integer $r$ larger than $R/2\pi$, then set $t_{j} = \bar{t}_{j} + 2\pi r(j-1)$, in such a way that $t_{k} - t_{k-1} > R$ for every $k=2,\ldots,2^{n-1}$.

Now consider the 
arithmetic mean of the $l$-th (complex) coordinates of  
$\f(t_1),\ldots,\f(t_{2^{n-1}})$. We show that this quantity is zero for $l = 2, \ldots, \tilde{n}$. Indeed, from the definition of $t_j$, we have
$$
\sum_{j=1}^{2^{n-1}} e^{i t_j \g_l} =
\prod_{k=1}^{n-1} \left(1+e^{i \pi \g_l / |\g_{k+1}|}\right)\,,
$$
which is zero since so is the $k$-th factor when $|\g_l| = |\g_{k+1}|$.

On the other hand, the arithmetic mean of the first coordinate is uniformly bounded away from zero. Indeed
\begin{align}
\left| \frac{1}{2^{n-1}}\sum_{j=1}^{2^{n-1}} e^{i \g_1t_j}\right|
&= \prod_{k=1}^{n-1} \left|\frac{1+e^{i \pi \g_1/ |\g_{k+1}|}}{2}\right|
= \prod_{k=2}^{n} \cos \left(\frac{\pi}{2 |\g_{k}|}\right)=  \exp\left({\sum_{k=2}^{n} \log\left(\cos\left(\frac{\pi}{2|\g_k|}\right)\right)}\right)\nonumber\\
&
\geq \exp\left({\sum_{k=2}^{\infty} \log\left(\cos\left(\frac{\pi}{2k}\right)\right)}\right)
 =\nu\,.\label{SF}
\end{align}
Since ${\log \left ( \cos \left (\frac{\pi}{2 k} \right )\right ) \sim -\frac{\pi^2}{8k^2}}$
as $k$ tends to infinity, then the sum $\sum_{k \geq 2} \log \left ( \cos
\left (\frac{\pi}{2 k} \right )\right )$ converges to a (negative) finite value $l$.
As a consequence, $\nu=\exp(l)$ is a positive number.

Therefore we have found a sequence of numbers $t_j$ such that the arithmetic mean of the first coordinate of
$\f(t_1),\ldots,\f(t_{2^{n-1}})$ is
uniformly bounded away from zero and the arithmetic means of following $\tilde{n}-1$ coordinates are zero.
According to~\eqref{eq:rotinv}, the role of $t_1,\dots,t_{2^{n-1}}$ can equivalently be played, for every $k\in \N$, by the
$2^{n-1}$-uple
$$
t_{j}^{k} = t_j + 2\pi m k\,, \quad j=1,\ldots,2^{n-1},
$$
where the integer $m$ is larger than $r+ t_{2^{n-1}}/2\pi$.
Now, let $l \in \{\tilde{n}+1,\dots,\K\}$, so that $\g_l \notin \Z$. For every $h \in \N$, the arithmetic mean of the $l$-th coordinate of the points $\f(t_{j}^{k})$ ($k=0,\dots,h$, $j=1,\ldots,2^{n-1}$)
is
\begin{align*}
 \frac{1}{2^{n-1}(h+1)}\sum_{j=1}^{2^{n-1}} \sum_{k=0}^{h} e^{it_j^k\g_l}
&=
\frac{1}{2^{n-1}(h+1)}\sum_{j=1}^{2^{n-1}} e^{it_j\g_l}\sum_{k=0}^{h} e^{i2\pi mk\g_l}\\
&=  \left(\frac{1}{2^{n-1}}\sum_{j=1}^{2^{n-1}} e^{it_j\g_l}\right)
\frac{1}{(h+1)} \frac{1 - e^{i2\pi m (h+1)\g_{l}}}{1-e^{i2\pi m\g_{l}}}
\stackrel{h\to\infty}{\longrightarrow}0\,.
\end{align*}
Therefore, we found a sequence of points in the convex hull of $\f([0,\infty))$
converging to
$({2^{1-n}}\sum_{j=1}^{2^{n-1}} e^{i \g_1t_j},0,\ldots,0)$.
The lemma follows from~\eqref{SF} and by rotation invariance (see~\eqref{eq:rotinv}).
\end{proof}
\begin{rem}\label{REM_valeur_nu}
In order to estimate $\nu$, notice that, for every $x$ in $(-1,1)$,
$
-\frac{{x}^{2}}{2}-\frac{{x}^{4}}{11}
\leq \log(\cos(x)).
$
Hence, taking $x=\frac{\pi}{2k}$ for $k\geq 2$,
$$\sum_{k=2}^{\infty} \log \left ( \cos
\left (\frac{\pi}{2 k} \right )\right )>-\frac{\pi^2}{8}\sum_{k=2}^{\infty} \frac{1}{k^2}-\frac{\pi^4}{176}\sum_{k=2}^{\infty} \frac{1}{k^4} =-\frac{{\pi}^{8}+240\,{\pi}^{4}-1980\,{\pi}^{2}}{15840}$$
 from which one deduces $\nu>\exp \left (-\frac{{\pi}^{8}+240\,{\pi}^{4}-1980\,{\pi}^{2}}{15840} \right ) >\frac{2}{5}$. Numerically, one finds $\nu \approx 0.430$.
\end{rem}

\subsection{An auxiliary system}\label{SEC_auxilliary_system}

Let $(A,B,U,\Phi)$ satisfy  $(\ass)$. 
With every  non-resonant connectedness chain $S$ for $(A,B,U,\Phi)$ and  every $n\in \mathbf{N}$ we associate the subset 
$$
S_n=\{(j,k)\in S\,|\,1\leq j, k \leq n, j\neq k\}
$$
of $S$ and the control system on $\mathbf{C}^{n}$
$$
\dot{x}=\nu |b_{jk}|\left (e^{i \theta} e_{jk}^{(n)} - e^{-i \theta} e_{kj}^{(n)}  \right )x,\eqno{({\Sigma}_n)}
$$
where  $\theta = \theta (t)\in \mathbb{S}^{1}$ and $(j,k) = (j(t),k(t)) \in S_{n}$ are piecewise constant controls. 
Recall that
$e_{jk}^{(n)}$
is the $n\times n$ matrix whose entries are all zero but the one of index $(j,k)$ which is equal to $1$
and that
$
\nu=\prod_{k=2}^{\infty} \cos \left (\frac{\pi}{2 k} \right )
$ 
(see~Lemma~\ref{LEM_convexification}).

The control system  $({\Sigma}_n)$ is linear in $x$. For every
$\theta$ in $\mathbb{S}^{1}$ and  every $1\leq j, k\leq n, j\neq k$, the
matrix $e^{i \theta} e_{jk}^{(n)}- e^{-i \theta}e_{kj}^{(n)}$  is skew-adjoint with zero trace.
Hence the control system  $({\Sigma}_n)$
leaves the unit sphere $\mathbf{S}^{n}$  of $\mathbf{C}^n$ invariant. In order to take advantage
of the rich Lie group structure of group of matrices, it is also possible to
lift this system in the group $SU(n)$, considering $x$ as a matrix.


\subsection{Existence of a $n$-connectedness chain} \label{sec:nconn}

Notice that system $(\Sigma_{n})$ cannot be controllable if $S$ is not a $n$-connectedness chain (see~\cite[Remark~4.2]{Schrod}).
This motivates the following proposition.

\begin{prop}\label{PRO_Galerkin_connected}
Let $(A,B,U,\Phi)$ satisfy $(\mathfrak{A})$. If there exists
a connectedness chain  for $(A,B,U,\Phi)$,
then there exists a bijection $\sigma :\mathbf{N}\rightarrow \mathbf{N}$
such that, setting $\hat{\Phi} = (\phi_{\sigma(k)})_{k\in\N}$ and 
$\hat{S} = \{(\sigma(j),\sigma(k))\,:\, (j,k) \in S\}$, $(A,B,U,\hat\Phi)$ satisfies $(\ass)$ and $\hat{S}$ is a $n$-connectedness chain for $(A,B,U,\hat{\Phi})$ for every $n$ in $\mathbf{N}$.
\end{prop}

\begin{proof}
Following \cite[Proof of Theorem~4.2]{genericity-mario-paolo}, $\sigma$ can be constructed recursively by setting $\sigma(1)=1$ and $\sigma(n+1)= \min \alpha \left ( \{ \sigma(1), \ldots,  \sigma(n) \}\right)$, where, 
for every subset $J$ of $\mathbf{N}$,
$
\alpha(J)=\left \{k\in \mathbf{N}\setminus J \mid \langle  \phi_j, B \phi_k \rangle \neq 0 \mbox{ for some } j \mbox{ in } J \right \}.
$
\end{proof}
Proposition~\ref{PRO_Galerkin_connected} proves the first part of the statement of Theorem~\ref{PRO_L1_estimates}.


\section{Modulus tracking}
\label{sec:modulus}

The aim of this section is to prove the following proposition, which is the main step in the proof of 
Theorem~\ref{THE_Control_collectively}.
\begin{prop}\label{PRO_tracking_modulus}
Let $\delta >0$.
Let $(A,B,[0,\delta], \Phi)$ satisfy $(\mathfrak{A})$ and admit a non-resonant connectedness chain. 
Then, for every continuous curve $\uni:[0,T] \rightarrow \mathbf{U}(\CH)$ such that $\uni_{0} = \Id$, $\nss $ in $\mathbf{N}$,
and $\eps>0$, there exist $T_u>0$, a continuous increasing bijection $s:[0,T]\rightarrow [0,T_u]$, and a piecewise constant control function $u:[0,T_u]\rightarrow [0,\delta]$ such that
the propagator $\pro^{u}$ of equation~\eqref{eq:main}
satisfies
$$
\big |  |\langle   \phi_{j} ,\uni_t \phi\rangle | - |\langle   \phi_{j},\pro^{u}_{s(t)}\phi \rangle |\big |< \eps \quad 
\mbox{ for every }  t \in[0,T],\  j \in \mathbf{N}, 
$$
for every $\phi\in\mathrm{span}\{\phi_1,\dots,\phi_\nss\}$ with $\|\phi\|=1$.

\end{prop}

The proof of Proposition~\ref{PRO_tracking_modulus} splits in several steps. In Section~\ref{SEC_classical_tracking_results} we recall some classical results of finite dimensional control theory, which, in Section~\ref{SEC_controlability_FGA}, are applied to system~$(\Sigma_{n})$ introduced in Section~\ref{SEC_auxilliary_system}. In Section~\ref{SEC_controllability_higher_dimension}  we prove that 
system~$(\Theta_{N})$ can track in projection the trajectories of system~$(\Sigma_{n})$.
Then we prove tracking for the original infinite dimensional system in Section~\ref{SEC_controllability_infinite_dimensional}.
The proof of Proposition~\ref{PRO_tracking_modulus} is completed in Section~\ref{sec:1234}.

\subsection{Tracking: definitions and general facts}\label{SEC_classical_tracking_results}

Let $M$ be a smooth manifold, $U$ be a subset of $\mathbf{R}$, and $f:M\times U \rightarrow TM$  be 
such that, for every $x$ in $M$ and every $u$ in $U$,  $f(x,u)$ belongs to $T_xM$ and $f(\cdot,u)$ is smooth. Consider the 
control system
\begin{equation}\label{EQ_sys_contr}
\dot{x}=f(x,u),
\end{equation}
whose admissible controls are piecewise constant functions $u:\mathbf{R}\to U$.
For a fixed $u$ in $U$, we denote by $f_u$ the vector field $x \mapsto f(x,u)$. 

\begin{defn}[Tracking]
Given a continuous curve $\mathbf{c}:[0,T]\rightarrow M$ we say that system~(\ref{EQ_sys_contr}) can \emph{track up to time reparametrization} the curve $\mathbf{c}$ if for every $\varepsilon >0$ there exist
$T_u>0$, an increasing bijection $s:[0,T]\rightarrow [0,T_u]$, and a piecewise constant control $u:[0,T_u]\rightarrow U$ such that the solution $x:[0,T_u]\rightarrow M$ of (\ref{EQ_sys_contr}) with control $u$ and initial condition $x(0)=\mathbf{c}(0)$ satisfies $\mathrm{dist}(x(s(t)),\mathbf{c}(t))<\varepsilon$ for every $t \in [0,T]$, where $\mathrm{dist}(\cdot,\cdot)$ is a fixed distance compatible with the topology of $M$. If $s$ can be chosen to be the identity then we say that system~(\ref{EQ_sys_contr}) can track $\mathbf{c}$ \emph{without time reparametrization}.
\end{defn}

Notice that this definition is independent of the choice of the distance $\mathrm{dist}(\cdot,\cdot)$.
Next proposition gives well-known sufficient conditions for tracking. 
It is a simple consequence of small-time local controllability (see for instance~\cite[Proposition~4.3]{Poisson} and~\cite{jaku}).

\begin{prop}\label{PRO_tracking}
If, for every $x$ in $M$,  $\{f(x,u) \mid u \in U\} = \{-f(x,u) \mid u\in U\}$ and $\Lie_x(\{f_u\mid  u \in U\})=T_xM$, then system~(\ref{EQ_sys_contr}) can track up to time reparametrization any continuous curve in $M$. 
\end{prop}

%

\subsection{Tracking in $(\Sigma_{n})$}\label{SEC_controlability_FGA}

We now proceed with the first step of the proof of Proposition~\ref{PRO_tracking_modulus}. Using Proposition~\ref{PRO_tracking} we can prove the following.

\begin{prop}\label{PRO_contr_dim_finie_Sigma}
Let $(A,B,U, \Phi)$ satisfy $(\mathfrak{A})$. Let $S$ be a non-resonant
connectedness chain for $(A,B,U,\Phi)$
 such that, for every $n$ in $\mathbf{N}$, 
 $S$ is a $n$-connectedness chain. 
Then, for every $n$ in $\mathbf{N}$, the finite dimensional control
system $({\Sigma}_n)$ can track up to time-reparametrization any curve 
in $SU(n)$.
\end{prop}
\begin{proof}
Recall that $
S_n=\{(j,k)\in S\,|\,1\leq j, k \leq n, j\neq k\}.
$
In order to apply Proposition~\ref{PRO_tracking} we notice that  the set
$$
\CV(x) 
= \left\{\nu |b_{jk}|\left (e^{i \theta} e_{jk}^{(n)} - e^{-i \theta} e_{kj}^{(n)}\right) x \,:\,  \theta \in \mathbb{S}^{1},  (j,k) \in S_n \right\} 
$$
is symmetric with respect to $0$ and we are left to prove that the Lie algebra generated by the linear vector fields
$x\mapsto \nu|b_{jk}| (e^{i \theta} e_{jk}- e^{-i \theta}e_{kj})x$
contains the whole tangent space  $\mathfrak{su}(n) x$ 
of the state manifold  $SU(n)$. 
The latter condition is verified if and only if 
$B^{(n)}$ is connected, as shown in the proof of Proposition~\ref{PRO_contr_dim_finie}.
\end{proof}

\subsection{Tracking trajectories of $(\Sigma_{n})$ in $(\Theta_{N})$}\label{SEC_controllability_higher_dimension}
 
Next proposition states that, for every $N\geq n$, system $(\Theta_{N})$, defined in Section~\ref{sec:conv}, can track without time reparametrization, in projection on the first $n$ components,  every trajectory of system~$(\Sigma_{n})$.

Hereafter we denote by 
$\Pi^{(N)}_{n}$ 
the projection  mapping  a $N \times N$ complex matrix to the  $n \times n$ matrix obtained by removing the last $N-n$ columns and the last $N-n$ rows.

\begin{prop}\label{PRO_tracking_elementaire}
Let $\delta > 0$. Let $(A,B,[0,\delta], \Phi)$ satisfy $(\mathfrak{A})$ and admit  a non-resonant
connectedness chain $S$.
For every  $n, N \in \mathbf{N}$, $N \geq n$, $\eps >0$ and for every trajectory $x:[0,T]\to SU(n)$ of system~$(\Sigma_{n})$ with initial condition $x(0) = I_{n}$ there exists a piecewise constant control $u:[0,T] \to [1/\delta, +\infty) $  such that 
the solution $y:[0,T] \to SU(N)$ of system~$(\Theta_{N})$ with initial condition $I_{N}$ satisfies
$$
\| x(t) - \Pi_n^{(N)} (y(t))\|<\eps \quad \mbox{ for every } t \in [0,T].
$$
\end{prop}

\begin{proof}
%
Given  a trajectory $x(t)$ of system~$(\Sigma_{n})$ with initial condition $x(0) = I_{n}$, denote by  $(j, k) = (j(t), k(t)) \in S_{n}$ and $\theta = \theta(t)\in \mathbb{S}^{1}$ its corresponding control functions. Being these functions piecewise constant, it is possible to  
write $[0,T] = \bigcup _{p=0}^{q} [t_{p}, t_{p+1}]$ in such a way that $j,k$, and $\theta$ are constant on $[t_{p}, t_{p+1})$ for every $p=0,\ldots,q$.

We are going to construct the control 
$u$
by applying recursively Lemma~\ref{LEM_convexification}.  
Let $\bar{\delta} > 1/\delta$.
Fix $p \in \{0,\ldots,q\}$ and $j,k, \theta$ such that $(j(t),k(t)) = (j,k)$ and $\theta(t)=\theta$ on $[t_{p},t_{p+1})$. Apply Lemma~\ref{LEM_convexification} with $\g_{1} = \lambda_{j} - \lambda_{k}$, 
$\{\g_{2}, \ldots,\g_{\K}\} = \{\lb_{l} - \lb_{m} \mid l,m \in \{1,\ldots,N\}, b_{lm}\neq0, \{l,m\}\neq\{j,k\},$ and $l\neq m \}$, 
$R = \max_{(l,m) \in S_{n}} \frac{|b_{ll} - b_{mm}|}{|\lb_{l} - \lb_{m}|} T + \bar{\delta} T$, and $t_0 = t_{0}(p)$ to be fixed later depending on $p$. Then, for every $\eta >0$, there exist 
$h = h(p) > 1/\eta$ and a sequence $(w_{\al}^{p})_{\al=1}^{h}$ such that $w_{1}^{p} \geq t_{0}$, $w_{\al}^{p}-w_{\al-1}^{p} > R$, and such that 
$$
\left|\frac{1}{h} \sum_{\al=1}^{h} e^{i(\lb_{k} - \lb_{j}) w^{p}_{\al}}  - \nu \frac{\bar{b}_{jk}}{|b_{jk}|} e^{i\theta}\right| < \eta,
$$
and
$$
\left|\frac{1}{h} \sum_{\al=1}^{h} e^{i(\lb_{l} - \lb_{m}) w^{p}_{\al} }\right| < \eta,
$$
for every $l,m \in \{1,\ldots,N\}$ such that $b_{lm}\neq0, \{l,m\}\neq\{j,k\},$ and $l\neq m$.

Set $\tau^{p}_{\al} = t_{p}+ (t_{p+1} - t_{p}){\al/h}$, $\al = 0,\ldots, h$, and define the piecewise constant function
\begin{equation}\label{eq:vht}
v_{\eta}(t) = \sum_{p=0}^{q} \sum_{\al=1}^{h} 
\left(w_{\al}^{p}  +  i \tau^{p}_{\al} \frac{b_{jj} - b_{kk}}{\lb_{j} - \lb_{k}} \right) \chi_{[\tau^{p}_{\al-1}, \tau^{p}_{\al})}(t)\,.
\end{equation}
 Note that by choosing $t_{0}(p) = w^{p-1}_{h(p-1)} + R$ for $p=1,\ldots,q$ and $t_{0}(0) = R$ 
we have that $v_{\eta}(t)$ is non-decreasing.

Set $M(t) =  \nu |b_{j(t)k(t)}| \left(e^{i \theta(t)} e^{(N)}_{j(t)k(t)} - e^{-i \theta(t)} e^{(N)}_{k(t)j(t)}\right)$.
From the construction of $v_{\eta}$ we have
\begin{equation}\label{eq:convergenza}
\int_{0}^{t} \vartheta_{N}(s,v_{\eta}(s)) ds
\stackrel{\eta \to 0}{\longrightarrow}
\int_{0}^t M(s) ds, 
\end{equation}
uniformly with respect to $t \in [0,T]$, where $\vartheta_{N}$ is defined as in~\eqref{eq:vartheta}.
This convergence guarantees (see for example~\cite[Lemma~8.2]{book}) that, denoting by 
$y_{\eta}(t)$ the solution  of system~$(\Theta_{N})$ with control $v_{\eta}$ and initial condition $I_{N}$,
$\Pi^{(N)}_{n}(y_{\eta}(t))$ converges to $x(t)$ as $\eta$ tends to $0$ uniformly with respect to  $t \in [0,T]$.
Hence, for every $\eta$ sufficiently small, 
$$
\|\Pi_n^{(N)} \circ y_{\eta}(t) -x(t)\|<\frac{\eps}{2} \quad \mbox{ for every } t \in [0,T].
$$

If the functions $v_{\eta}$ 
were of the type $t \mapsto \int_0^t u(s) ds$ for some $u:[0,T] \to [1/\delta, \infty)$ piecewise constant,
then we would be done.
For every $\eta > 0$ consider
the piecewise linear continuous function $\check{v}_{\eta}$ uniquely defined on every interval $[t_{p}, t_{p+1})$ by
$$
\left\{
\begin{array}{rcll}
\check{v}_\eta(t_{p})&=& w^{p-1}_{h} & \\
\ddot{\check{v}}_{\eta}(t)&=&0&\mbox{if }t\in\bigcup_{\al=1}^h [\tau^{p}_{\al-1},\tau^{p}_{\al-1}+\frac{(t_{p+1} - t_{p})}{h^2}),
 \\
\check{v}_\eta(\tau^{p}_{\al-1}+\frac{(t_{p+1} - t_{p})}{h^2})&=&w^{p}_\al&\mbox{for }\al=1,\dots,h,\\
\dot{\check{v}}_{\eta}(t)&=&\bar\delta&\mbox{if }t\in\bigcup_{\al=1}^h [\tau^{p}_{\al-1}+\frac{(t_{p+1} - t_{p})}{h^2},\tau^{p}_{\al}),\end{array}
\right.
$$
where we set $w^{-1}_{h} = 0$
(see Figure~\ref{fig:sequence}).
\begin{figure}
\label{fig:sequence}
\begin{center}
\includegraphics[width=12cm]{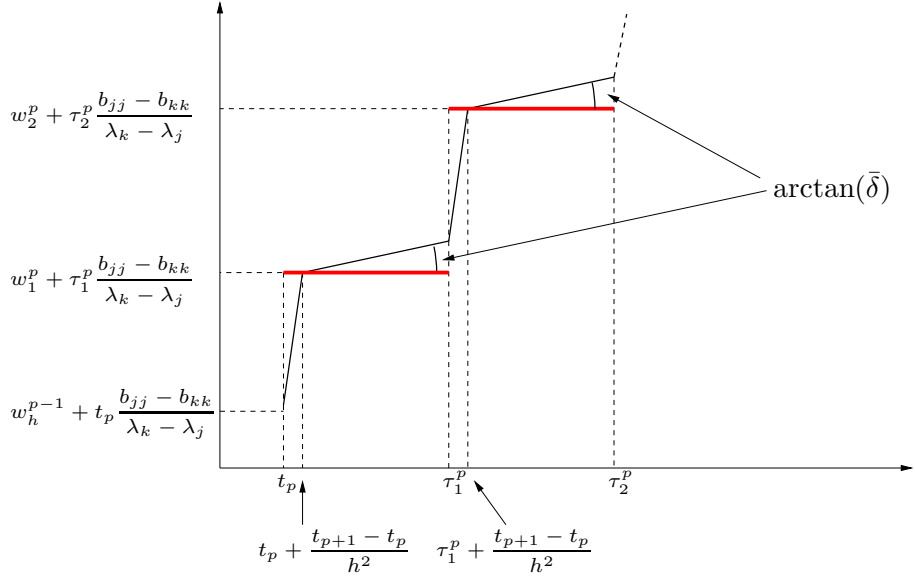}
\caption{The piecewise constant function $v_{\eta}$ (in bold) and the piecewise linear approximation $\check{v}_\eta$ with slope greater than $\bar\delta$.} 
\end{center}
\end{figure}
On each interval $ [\tau^{p}_{l-1}+\frac{(t_{p+1} - t_{p})}{h^2},\tau^{p}_{l})$ the difference between $v_{\eta}$ and $\check{v}_\eta$ is bounded in absolute value by $\bar\delta (t_{p+1} - t_{p})/h $.
Therefore,
$$
\sup\left\{\left\|
\vartheta_{N}(t,v_{\eta}(t)) - 
\vartheta_{N}(t, \check{v}_{\eta}(t))
\right\|
\mid t\in\bigcup_{l, p} \left[\tau^{p}_{l-1}+\frac{(t_{p+1} - t_{p})}{h^2},\tau^{p}_{l}\right)\right\}
$$
tends to zero
as $\eta$ tends to $0$.

Since $\|\vartheta_{N}(t,v)\|$ is uniformly bounded with respect to $(t,v) \in [0,T]\times\R$ and the measure of  
$\bigcup_{l,p} [\tau^{p}_{l-1},\tau^{p}_{l-1}+\frac{(t_{p+1} - t_{p})}{h^2})$ goes to $0$ as $\eta$ goes to $0$, we have
$$
\int_{0}^t \left( 
\vartheta_{N}(\tau, v_{\eta}(\tau)) -\vartheta_{N}(\tau, \check{v}_{\eta}(\tau))
 \right) 
 d\tau \stackrel{\eta \to 0}{\longrightarrow}0\quad\mbox{ uniformly with respect to }t \in[0,T]\,.
$$
In particular, for  $\eta$ sufficiently small, if $\check{y}_{\eta}$ denotes the solution of system~$(\Theta_{N})$ with control $\check{v}_\eta$ and initial condition $I_{N}$, then
$$
\|y_{\eta}(t) - \check{y}_{\eta}(t)\| < \frac\eps 2 \quad \mbox{ for every } t \in [0,T].
$$
Finally, $u$ can be taken
as the derivative of $\check{v}_{\eta}$, which is defined almost everywhere.
\end{proof}

\subsection{Tracking trajectories of  $(\Sigma_{n})$ in the original system}\label{SEC_controllability_infinite_dimensional}

Next proposition extends the tracking property obtained in the previous section from the system $(\Theta_{N})$ to the infinite dimensional system~\eqref{EQ_main_reparam}.
We denote by $\Pi_{n}: \CH \to \C^{n}$ the projection mapping 
$\psi  \in \CH$ to $(\langle \phi_{1}, \psi \rangle, \ldots, \langle \phi_{n}, \psi\rangle) \in \C^{n}$ and we write $\phi^{(n)}_{k}$ for $\Pi_{n}\phi_{k}$. 

\begin{prop}\label{PRO_tracking_dim_inf}
Let $\delta > 0$ and let  $(A,B,[0,\delta], \Phi)$ satisfy $(\ass)$.
For every $\eps>0$, $n \in  \mathbf{N}$, and for every trajectory $x:[0,T] \to SU(n)$ of system~$(\Sigma_{n})$ with initial condition $I_{n}$ there exists a piecewise constant function $u:[0,T] \rightarrow [1/\delta, +\infty)$ such that the propagator $\pro^{u}$ of~\eqref{EQ_main_reparam} satisfies 
$$
\big|  | \langle \phi^{(n)}_{j} , x(t) \Pi_n\phi\rangle | - | \langle \phi_j , \pro^{u}_{t} (\phi)\rangle | \big| <\eps
$$
 for every $\phi\in\mathrm{span}\{\phi_1,\dots,\phi_n\}$ with $\|\phi\|=1$ and every  
 $t$ in $[0,T]$ and  $j$ in $\N$.
\end{prop}

\noindent {\it Proof.} Consider $\mu > 0$.
 For every $j\in\N$ the hypothesis that $\phi_j$ belongs to $D(B)$ implies that
the sequence $(b_{jk})_{k\in \N}$ is in $\ell^2$. 
It is therefore possible to choose
$N\geq n $ such that
$\sum_{k> N}|b_{jk}|^2<\mu$ for every $j=1,\dots, n$. By Proposition~\ref{PRO_tracking_elementaire}, for every $\eta > 0$ and  
for every trajectory $x:[0,T]\to SU(n)$ of system~$(\Sigma_{n})$ with initial condition $
 I_{n}$, there exists a piecewise constant control $u^{\eta}:[0,T] \to [1/\delta, +\infty) $ such that 
the solution $y^{\eta}$ of system~$(\Theta_{N})$ with initial condition $I_{N}$ satisfies
$$
\| x(t) - \Pi_{n}^{(N)} (y^{\eta}(t))\|<\eta.
$$

Denote by $R^{\eta}(t,s):\C^{N} \to \C^{N}, 0\leq s,t \leq T$, the resolvent of system $(\Theta_{N})$ associated with the control $v^{\eta}(t) = \int_{0}^{t}u^{\eta}(\tau) d\tau$, so that $y^{\eta}(t) = R^{\eta}(t,0)$.
Fix $\phi\in\mathrm{span}\{\phi_1,\dots,\phi_n\}$ with $\|\phi\|=1$ and set
$$
Q^{\eta}(t) = e^{-v^{\eta}(t) A -t \diag(B)}  \pro^{u^{\eta}}_{t} (\phi).
$$
The components of $Q^{\eta}(t)$, say
$$
q^{\eta}_j(t)=e^{- i \lambda_j v^{\eta}(t) - t b_{jj}}\langle \phi_{j}, \pro^{u^{\eta}}_{t} (\phi)\rangle, \quad j\in \N\,,
$$
satisfy, for almost every $t\in [0,T]$,
 \begin{equation}\label{eq:qdot}
 \dot q_j^{\eta}(t)=\sum_{k=1}^\infty b_{jk}e^{i(\lambda_k-\lambda_j)v^{\eta}(t) + t (b_{kk} - b_{jj})}q_k^{\eta}(t).
 \end{equation}
Therefore
$Q^{\eta}_{N}(t) = \Pi_{N}Q^{\eta}(t) = (q^{\eta}_1(t), \ldots,q^{\eta}_N(t))^{T} $ satisfies the time-dependent linear equation
$$
\dot Q^{\eta}_{N}(t) = \vartheta_{N}(t,v^{\eta}(t)) Q^{\eta}_{N}(t)  + P^{\eta}_{N}(t),
$$
where
$P^{\eta}_{N}(t) = 
 ( \sum_{k>N} b_{1k} e^{i(\lb_{k} - \lb_{1})v^{\eta}(t) + t (b_{kk} - b_{11})} q^{\eta}_{k}, \ldots, \sum_{k>N} b_{Nk} e^{i(\lb_{k} - \lb_{N})v^{\eta}(t) + t (b_{kk} - b_{NN})} q^{\eta}_{k} )^{T}$.
Hence
$$
Q^{\eta}_{N}(t)  = R^{\eta}(t,0) \Pi_N\phi + \int_{0}^{t} R^{\eta}(s,t)  P^{\eta}_{N}(s) ds.
$$
Consider the projection of the equality above on the first $n$ coordinates. 
Notice that, because of the choice of $N$, the norm of the first $n$ components of $P^{\eta}_{N}(t)$ is smaller than $\sqrt{\mu n}$.
By~\eqref{eq:convergenza}, $R^{\eta}(s,t)$ converges uniformly, as $\eta$ tends to $0$, to a time-dependent operator from $\C^{N}$ into itself which preserves the norm of the first $n$ components. Then, there exists $\eta$ sufficiently small such that
$$
\left\|  \Pi^{(N)}_{n} \left(\int_{0}^{t} R^{\eta}(s,t)  P^{\eta}_{N}(s) ds \right) \right\| < 2T\sqrt{\mu n} .
$$
Hence
\begin{align}
\| \Pi_{n} Q^{\eta}(t) - x(t) \Pi_n\phi\| 
&\leq  \| \Pi_{n} Q^{\eta}(t)  -  \Pi_{n}^{(N)} (R^{\eta}(t,0)) \Pi_n\phi\| + 
\|\Pi_{n}^{(N)} (R^{\eta}(t,0) )\Pi_n\phi - x(t) \Pi_n\phi\| \nonumber \\
& \leq 2T\sqrt{\mu n} + \eta< \frac \eps 2,\label{eq:normapiccola}
\end{align}
if $\mu <\eps^{2}/(32nT^{2})$ and  $\eta < T \sqrt{\mu n }$. In particular, for $j=1,\ldots,n$,
$$
\big| | \langle \phi^{(n)}_{j} , x(t)\Pi_n \phi\rangle | - | \langle \phi_j,  \pro^{u^{\eta}}_{t} (\phi) \rangle | \big|  < \frac \eps 2.
$$
It remains to prove the statement for $j>n$. From~(\ref{eq:normapiccola}) it follows
$$
\sum_{j=1}^{n} |q_{j}^{\eta}(t)|^{2} >  \left(1-\frac \eps 2 \right)^{2},
$$
then, since $ \pro^{u}_{t} $ is a unitary operator for every $t$, we have
$$
\sum_{j>n}  |q_{j}^{\eta}(t)|^{2} < 1 - \left(1-\frac \eps 2\right)^{2} < \eps\,. \eqno{\Box}
$$

\subsection{Proof of modulus tracking}\label{sec:1234}

The following proposition allows to reduce the  tracking problem stated in Proposition~\ref{PRO_tracking_modulus}
to the tracking of a curve in $SU(n)$.

\begin{prop}\label{PRO_lemme_exist_Galerkin}
For every continuous curve $\uni:[0,T]\rightarrow \mathbf{U}(\CH)$,  $\eps>0$, and $\nss \in \N$, there exist $n \geq \nss$ and a continuous curve $\mathbf{F}^{n}:[0,T]\rightarrow SU(n)$ such that 
$| \langle \phi_{j},\uni_t \phi_k \rangle  -  \langle \phi^{(n)}_{j},\mathbf{F}^{n}(t) \phi^{(n)}_k \rangle| <\eps$ for every $t$ in $[0,T]$, $1 \leq k\leq \nss$, and $j\in\N$.
\end{prop}

\begin{proof}
For every $n$ in $\N$, define the function $g_{n}:t\mapsto \sum_{k=1}^{\nss}\sum_{l=1}^n  |\langle\phi_l, \uni_t \phi_k  \rangle|^{2}$. 
The functions $g_{n}$ are continuous and $g_{n}(t)$ converges monotonically to $\nss$ as $n$ tends to infinity for every $t \in [0,T]$. Hence $g_{n}$ converges to the constant function $\nss$ uniformly with respect to $t \in [0,T]$.
Therefore,
for every $1\leq k \leq \nss$, $\psi_{k}^{n}(t) = \sum _{l=1}^{n}\langle \phi_{l}, \uni_{t} \phi_{k} \rangle \phi_{l}$ converges to $\uni_{t} \phi_{k}$ uniformly with respect to $t \in [0,T]$. In particular,  the matrix
$\left ( \langle \psi_{k}^{n}(t), \psi_{j}^{n}(t) \rangle \right)_{j,k=1}^{r} $
converges to $I_{r}$ uniformly with respect to $t \in [0,T]$.

For every $n$ large enough, the vectors $\psi_{1}^{n}(t),\ldots,\psi_{\nss}^{n}(t)$ are linearly independent and can be completed to a basis 
$\mathcal{B}^{n}(t) = (\psi_{1}^{n}(t),\ldots,\psi_{\nss}^{n}(t), \varphi_{\nss+1}^{n}(t),\ldots,  \varphi_{n}^{n}(t))$ of $\mathrm{span}\{\phi_{1},\ldots,\phi_{n}\}$ depending continuously on  $t\in [0,T]$.
Let $M^{n}(t)$ be the $n \times n$ matrix of the components of $\mathcal{B}^{n}(t)$ with respect to the basis 
$(\phi_{1},\ldots,\phi_{n})$. Denote by $\mathbf{F}^{n}(t)$ the matrix whose columns are the
 Gram--Schmidt transform of the columns of $M^{n}(t)$. Then, for every $n$ large enough, $F^{n}$ satisfies the statement of the proposition.
\end{proof}

We are now ready to prove Proposition~\ref{PRO_tracking_modulus}.

\begin{proof}[Proof of Proposition~\ref{PRO_tracking_modulus}]
Thanks to Proposition~\ref{PRO_reparam}, it is sufficient to prove
that, for every 
$\nss $ in $\mathbf{N}$, $\eps>0$, 
  and every
continuous curve $\uni:[0,T] \rightarrow {\mathbf U}(\CH)$,  there exist $T_u>0$, a continuous increasing bijection $s:[0,T]\rightarrow [0,T_u]$, and a piecewise constant control function 
$u:[0,T_u]\rightarrow [1/\delta,+\infty)$ such that
the propagator $\pro^{u}$ of system~\eqref{EQ_main_reparam} satisfies 
$$
\big |  |\langle \phi_{j}, \uni_t \phi \rangle | - |\langle  \phi_{j}, \pro^{u}_{s(t)}\phi\rangle | \big |< \eps\,,
$$ 
 for every  $\phi\in\mathrm{span}\{\phi_1,\dots,\phi_\nss\}$ with $\|\phi\|=1$
 and every 
 $t \in[0,T]$, $j \in \mathbf{N}$.

By Proposition~\ref{PRO_lemme_exist_Galerkin} there exist $n \geq \nss$ and
a continuous curve $\mathbf{F}^n:[0,T]\rightarrow SU(n)$ such that 
$ |\langle \phi_{j}, \uni_t \phi_k \rangle  - \langle \phi^{(n)}_{j},\mathbf{F}^n(t) \phi^{(n)}_k \rangle | <\eps/3$ for every $t$ in $[0,T]$, $1\leq k \leq \nss$, and $j\in \N$.

By Proposition~\ref{PRO_contr_dim_finie_Sigma}, 
there exists an admissible trajectory $x:[0,T_{\Sigma}] \to  SU(n)$ of system $({\Sigma}_{n})$ 
with initial condition $I_{n}$ satisfying
 $$
 \|  x(s(t))  - \mathbf{F}^n(t)\| < \frac\eps 3 \quad\mbox{ for every } t \in [0,T].
 $$

Finally, by Proposition~\ref{PRO_tracking_dim_inf}, there exists 
a piecewise constant function $u:[0,T_{\Sigma}]\rightarrow [1/\delta, +\infty)$ such that the propagator $\pro^{u}$ of~\eqref{EQ_main_reparam} satisfies 
$$
\big| | \langle \phi^{(n)}_{j} , x(t) \Pi_n\phi\rangle | -  | \langle \phi_j, \pro^{u}_{t} (\phi) \rangle |  
\big| <\frac \eps 3
$$
for every  $\phi\in\mathrm{span}\{\phi_1,\dots,\phi_\nss\}$ with $\|\phi\|=1$ and every $t \in [0,T_{\Sigma}]$,  $j\in \N$.
\end{proof}

\subsection{Estimates of the $L^1$ norm of the control}\label{sec:L1}

We derive now estimates of the minimal $L^1$ norm of the control $u$ whose existence is asserted in Proposition~\ref{PRO_tracking_modulus}. We focus here on the physically relevant transitions inducing permutations between eigenvectors of $A$.

The strategy to get $L^1$ estimates is the following. 
Recall that, instead of considering the control system $\dot{x}=(A+uB)x$ driven by a  piecewise continuous function $u:[0,T_u]\rightarrow [0,\delta]$, we have defined the function ${\cal P}(u):[0,\|u\|_{L^1}]\rightarrow [1/\delta,\infty)$ and considered the control system $\dot{x}=({\cal P}(u)A+B)x$. 
By Propositions~\ref{prop:4321} and \ref{PRO_reparam}, in order to estimate the $L^1$ norm of $u$, it is enough to estimate the time needed to transfer the system $\dot{x}=(uA+B)x$ from a given source to an $\eps$-neighborhood of a given target. We observe that the time needed to transfer $\dot{x}=(uA+B)x$ from one state to an $\eps$-neighborhood of another is smaller than or equal to the time needed 
to transfer system $(\Sigma_n)$ 
between the $n$-Galerkin approximations of the initial and the final condition for $n$ large enough. 

We proceed to the proofs of Proposition~\ref{prop:L1estimatessimple} and Theorem~\ref{PRO_L1_estimates}. 

\begin{proof}[Proof of Proposition~\ref{prop:L1estimatessimple}]
Let $S$ be a non-resonant chain of connectedness for $(A,B,[0,\delta],\Phi)$.  Choose $(j,k)$ in $S$ and let $m$ in $\mathbf{N}$ be such that $m\geq j ,k$. The solution $x:\mathbf{R}\rightarrow SU(m)$ of the Cauchy problem 
$$
\dot{x}= \nu   |b_{jk}|\left ( e_{jk}^{(m)}- e_{kj}^{(m)}  \right ) x,\quad\quad x(0)=I_m
$$
is a trajectory of $(\Sigma_m)$ and has the form $x(t)=\exp\left(t\nu   |b_{jk}|\left ( e_{jk}^{(m)}- e_{kj}^{(m)}  \right )\right)$. The matrix
 $M:=x \left ( \frac{\pi}{2 \nu |b_{jk}|}\right )$ satisfies $M \phi_l^{(m)} = \phi_l^{(m)}$ if $l\neq j,k$, $M \phi_j^{(m)}= -\phi_k^{(m)}$,  $M \phi_k^{(m)}= \phi_j^{(m)}$.
In other words, the control system $(\Sigma_m)$ can exchange (up to a phase factor) the eigenstates $j$ and $k$ of $A^{(m)}$, leaving all the others eigenstates invariant, in time 
$\frac{\pi}{2 \nu |b_{jk}|}$.
\end{proof}

\begin{proof}[Proof of Theorem~\ref{PRO_L1_estimates}]
 
Using Proposition~\ref{PRO_Galerkin_connected}  one may assume that $S$ is a $m$-connectedness chain for every $m \in \N$. 
 Let us prove by induction that every permutation of $\{1,\ldots,m\}$ is a product of at most $2^{m-1}-1$ transpositions of the form $(j~k)$ with  $(j,k)$ in $S \cap \{1,\ldots,m\}^{2}$. 
Let $h(n)$ be the minimal integer such that every permutation of $\{1,\ldots ,n\}$ is the product of at most
$h(n)$ transpositions of the  form $(j~k)$ with  $(j,k)$ in $S_m$. 
 For every permutation $\sigma$ of $\{1,\ldots ,n+1\}$, either $\sigma(n+1)=n+1$, and $\sigma$ is generated by at most $h(n)$ transpositions, or $\sigma(n+1) <n+1$. In this case, there exists $1 \leq k\leq n$ such that $(k,n+1) \in S$. Since the product $(k~n+1)(k~\sigma(n+1)) \sigma$ leaves $n+1$ invariant, it is a product of at most $h(n)$ permutations. As a conclusion, $h(n+1)\leq 2 h(n) +1$ and since $h(2)=1$, we find $h(m)\leq 2^{m-1}-1$.

The time needed for each of the transpositions $(j~k)$ with $(j,k)$ in $S$ has been computed in Proposition~\ref{prop:L1estimatessimple}. The conclusion follows from the estimate $\nu > 2/5$ proved in Remark~\ref{REM_valeur_nu}.
 \end{proof}

\begin{rem}
 The bound given in Theorem~\ref{PRO_L1_estimates} does not depend on $\eps$. However, it is possible that the time $T_u$ needed to achieve the transfer of system~\eqref{eq:main} grows to infinity as $\eps$ tends to zero.
\end{rem}

\begin{rem}
 Theorem~\ref{PRO_L1_estimates} could be stated in a more general way. Indeed the result~\cite[Theorem~6.2]{sachkov} gives the existence of a uniform bound on the time needed to steer system~$(\Sigma_{n})$ from any linear combination of the first $n$ eigenstates to any other. This fact guarantees the existence of a uniform bound for the $L^{1}$-norm of a control steering system~\eqref{eq:main}  from any linear combination of the first $n$ eigenstates to any neighborhood of any other unitarily equivalent linear combination of the first $n$ eigenstates.
 Such time estimates have been given explicitly in the case $S=\mathbf{N}^2$ in~\cite[Section 5]{agrachev-chambrion}.  This result could be generalized   to the case under consideration. It is, however, rather technical and involves advanced notions of Lie group theory. 
\end{rem}

Following the method of \cite[Section 4.5]{Schrod}, one can also give a lower bound for the $L^1$ norm of the control.

\begin{prop}\label{PRO_L1_lower_estimates}
Let $(A,B,U,\Phi)$ satisfy $(\ass)$.
For every $\uni$ in $\mathbf{U}(\cal H)$, $m$ in $\mathbf{N}$, $\eps>0$, $\theta_{1},\ldots,\theta_{n} \in \R$, and every piecewise constant function $u:[0,T_u]\rightarrow U$ such that the propagator $\pro^{u}$ of \eqref{eq:main} satisfies $\|e^{i \theta_k} {\uni}(\phi_k) - \pro^{u}_{T_u}(\phi_k)\|\leq \eps$ for $1 \leq k \leq m$, one has
$$
\|u\|_{L^1} \geq \sup_{1\leq k \leq m} \sup_{j \in \mathbf{N}} \frac{\big  \langle \phi_j, \phi_k \rangle - |\langle \phi_j,\uni \phi_k\rangle |\big|-\eps}{\| B \phi_j\|}.
$$
\end{prop}

Notice that while some strong assumptions about the existence of connectedness chains are needed in Theorem~\ref{PRO_L1_estimates}, Proposition~\ref{PRO_L1_lower_estimates} is valid even if~\eqref{eq:main} is not approximately controllable.

\section{Phase tuning}
\label{sec:phase}

Based on  
Proposition~\ref{PRO_tracking_modulus}, 
we shall now 
complete the proof of Theorem~\ref{THE_Control_wave_function}
proving  approximate simultaneous controllability (see Proposition~\ref{prop:phasetuning} below).

In order to outline the mechanism of the proof, we treat in a first time the case of a single wave function (proving directly the first part of Corollary~\ref{THE_Control_density_matrices}) and we then turn, in Section~\ref{phases-2}, to the general case. 

\subsection{Phase tuning for the control of a single wave function}\label{phases-1}

Simultaneous controllability is obtained from Proposition~\ref{PRO_tracking_modulus} applied both to \eqref{eq:main} and to its time-reversed version. 
If $(A,B,[0,\delta],\Phi)$
satisfies $(\mathfrak{A})$ and admits a non-resonant connectedness chain, then the same is true for $(-A,-B,[0,\delta],\Phi)$. 
{Notice, moreover, that, by unitarity of the evolution of the \Sch\ equation, 
if 
$u:[0,T]\to [0,\delta]$ steers $\psi_0$ to a $\eps$-neighborhood of
$\psi_1$ for the time-reversed control system
\begin{equation}\label{sch-tr}
\frac{d\psi}{dt}(t)=-(A+u(t)B) \psi(t), \quad u(t) \in [0,\delta],
\end{equation}
 then 
$u(T-\cdot):[0,T]\to [0,\delta]$
steers 
$\psi_1$ $\eps$-close to $\psi_0$ for the original system \eqref{eq:main}.

Take any eigenvector $\phi_{\bar k}$ such that $\lambda_{\bar k}\ne 0$ (its existence clearly follows from the existence of a non-resonant connectedness chain) 
and consider the control $u:[0,T]\to [0,\delta]$ steering $\psi_0$ to a $\eps$-neighborhood of 
$e^{i\theta}\phi_{\bar k}$ for some $\theta\in [0,2\pi)$.
The existence of such a $u$ follows from Proposition~\ref{PRO_tracking_modulus}, with  
$\uni$ any continuous curve in $\mathbf{U}(\CH)$ 
from the identity to a unitary operator sending $\psi_0$ into $\phi_{\bar k}$ and $\nss$ sufficiently large.

Similarly, there exist $\tilde u:[0,\tilde T]\to [0,\delta]$ and $\tilde \theta\in [0,2\pi)$ such that $\tilde u$ steers 
$\psi_1$ 
 $\eps$-close to $e^{i \tilde \theta}\phi_{\bar k}$ for \eqref{sch-tr}. 
Let 
$\tau>0$ be such that 
$$e^{\tau A}(e^{i \theta}\phi_{\bar k})=e^{i \tilde \theta}{\phi_{\bar k}}.$$
Hence, the concatenation of $u$, of the control constantly equal to zero for a time $\tau$, and of
 $\tilde u(\tilde T-\cdot)$, steers $\psi_0$ $2\eps$-close to $\psi_1$.

\subsection{Phase tuning for simultaneous control}\label{phases-2}
Let $\ns$ be the number of equations that  we would like to control simultaneously, as in Definition~\ref{DEF_simultaneous_controllability}.
The scheme of the argument is similar to the one above. The pivotal role of the orbit of $\{e^{t A}\phi_{\bar k}\mid t\}$  is now played by a torus of dimension $\ns$. 

The crucial point is to ensure that an orbit of $A$ ``fills'' the torus densely enough. 
This is formally stated in the proposition below.
Recall that a subset $\Omega_1$ is $\eps$-dense in a metric space $\Omega_2$ if any point of $\Omega_2$ is at distance smaller than $\eps$ from every point of $\Omega_1$. 
For every $m\in\N$ and $k_1,\dots,k_m\in\N$, define
\begin{align*}
{\cal T}({k_1},\ldots,{k_m})&=\{e^{\theta_1 A}\phi_{k_1}+ \cdots +e^{\theta_m A}\phi_{k_m}\mid \theta_1,\ldots,\theta_m\in\R\}\nonumber\\
{\cal C}({k_1},\ldots,{k_m})&=\{e^{t A}\phi_{k_1}+ \cdots +e^{t A}\phi_{k_m}\mid t\in\R\}.\nonumber
\end{align*}
Notice that, if $k_1,\dots,k_m$ are distinct,  
${\cal T}({k_1}\ldots,{k_m})$
is diffeomorphic to the torus $\mathbb{T}^m$.

\begin{prop}\label{prop:phasetuning}
Let 
$(A,B,[0,\delta],\Phi)$
satisfy $(\mathfrak{A})$ and admit a non-resonant connectedness chain.
Assume 
that, for every $\eta>0$ and $\ns$ in $\mathbf{N}$, 
there exist $\ns$ pairwise distinct positive integers $\bar k_1,\dots,\bar k_\ns$ such that 
$
{\cal C}({\bar k_1},\ldots,{\bar k_\ns})$ is $\eta$-dense in ${\cal T}({\bar k_1},\ldots,{\bar k_\ns})$.
Then~\eqref{eq:main} is simultaneously approximately controllable. 
\end{prop}
\begin{proof}
Take $\ns$ orthonormal initial conditions $\psi_0^1,\ldots,\psi_0^\ns$ and $\ns$ orthonormal final conditions  $\psi_1^1,\ldots,\psi_1^\ns$. 
Fix 
a tolerance $\eta>0$. 
Take $\bar k_1,\dots,\bar k_\ns$ as in the statement of the proposition. 

According to Proposition~\ref{PRO_tracking_modulus} (with $\nss$ sufficiently large and 
$\uni$ a continuous curve in $\mathbf{U}(\CH)$ 
from the identity to a unitary operator sending $\psi_0^j$ to $\phi_{\bar k_j}$ for each $j=1,\dots,\ns$), there exists a control $u(\cdot)$
 steering simultaneously each $\psi_0^j$, for  $j=1,\ldots,\ns$,    $\eta$-close to 
 $e^{\theta_jA}\phi_{\bar k_j}$ for some $\theta_1,\dots,\theta_\ns\in\R$.
 Similarly, 
 applying Proposition~\ref{PRO_tracking_modulus} to the triple $(-A,-B, [0,\delta],\Phi)$, 
 there exists a control  $\tilde u:[0,\tilde T]\to [0,\delta]$
 steering simultaneously  $\psi_1^j$
   $\eta$-close to $e^{\tilde \theta_jA}\phi_{\bar k_j}$ for \eqref{sch-tr},   for $j=1,\ldots,\ns$ and for some $\tilde \theta_1,\dots,\tilde \theta_\ns\in\R$.

Since the positive orbit of $A$ passing through $\sum_{j=1}^\ns e^{\theta_jA}\phi_{\bar k_j}$ is $\eta$-close to $\sum_{j=1}^\ns e^{\tilde \theta_jA}\phi_{\bar k_j}\in{\cal T}({\bar k_1},\ldots,{\bar k_\ns})$, then the concatenation of $u(\cdot)$, a control constantly equal to zero on a time interval of suitable length, and $\tilde u(\tilde T-\cdot)$ steers each $\psi_0^j$ $3\eta$-close to 
 $\psi_1^j$ for
 $j=1,\ldots,\ns$. 
\end{proof}

\begin{rem}
As it follows from the proof above, in Proposition~\ref{prop:phasetuning}  the hypothesis of existence of a non-resonant chain of connectedness can be replaced by the weaker hypothesis that both~\eqref{eq:main} and~\eqref{sch-tr} are simultaneously controllable up to phases in the sense of Proposition~\ref{PRO_tracking_modulus}.
\end{rem}

We are left to prove the following. 
\begin{lem}\label{lem:1234}
If $A$ has infinitely many distinct eigenvalues then
for every $\eta>0$ and $\ns \in \mathbf{N}$, 
there exist $\ns$ distinct positive integers $\bar k_1,\dots,\bar k_\ns$ such that 
${\cal C}({\bar k_1},\ldots,{\bar k_\ns})$ is $\eta$-dense in ${\cal T}({\bar k_1},\ldots,{\bar k_\ns})$. 
\end{lem}

We split the proof of Lemma~\ref{lem:1234} in four cases.  

\subsubsection{$\mathrm{dim}_\Q\left(\mathrm{span}_\Q(\lambda_k)_{k\in\N}\right)=\infty$}
This case is trivial: it is enough to take $\bar k_1,\dots,\bar k_\ns$ in such a way that $\lambda_{\bar k_1},\dots,\lambda_{\bar k_\ns}$ are $\Q$-linearly independent. Then  
${\cal C}({\bar k_1},\ldots,{\bar k_\ns})$ is dense (hence, $\eta$-dense for every $\eta>0$) in ${\cal T}({\bar k_1},\ldots,{\bar k_\ns})$.

\subsubsection{The spectrum of $A$ is unbounded}\label{sss-unbounded}
The most physically relevant case is the one in which the sequence $(\lambda_k)_{k\in\N}$ is unbounded.

Fix $\bar k_1$ such that $\lambda_{\bar k_1}\neq0$.  Take then $\bar k_2$ such that $|\lambda_{\bar k_2}|\gg|\lambda_{\bar k_1}|$ in such a way that the orbit 
${\cal C}({\bar k_1},{\bar k_2})$ is $\eta$-dense in  ${\cal T}({\bar k_1},{\bar k_2})$.
By recurrence, taking  $|\lambda_{\bar k_m}|\gg|\lambda_{\bar k_{m-1}}|$ for $m=2,\ldots,\ns$ we have that ${\cal C}({\bar k_1},\dots,{\bar k_m})$ is $\eta$-dense in  ${\cal T}({\bar k_1},\ldots,{\bar k_m})$. Indeed, by the recurrence hypothesis,
${\cal C}({\bar k_1},\dots,{\bar k_{m-1}})+ {\cal T}({\bar k_m})$ is $\eta$-dense in ${\cal T}({\bar k_1},\ldots,{\bar k_m})={\cal T}({\bar k_1},\dots,{\bar k_{m-1}})+ {\cal T}({\bar k_m})$ and the choice of $\lambda_{\bar k_m}$ is such that ${\cal C}({\bar k_1},\dots,{\bar k_m})$
is $\eta'$-dense in ${\cal C}({\bar k_1},\dots,{\bar k_{m-1}})+ {\cal T}({\bar k_m})$ for $\eta'$ arbitrarily small.

\subsubsection{The spectrum of $A$ is  bounded and $\mathrm{dim}_\Q\left(\mathrm{span}_\Q(\lambda_k)_{k\in\N}\right)=1$}\label{sss-oneQ}

The existence of a connectedness chain implies that there exist infinitely many pairwise distinct gaps between eigenvalues of $A$. Hence, the set of eigenvalues of $A$ has infinite cardinality. 

Since all the eigenvalues of $A$ are $\Q$-linearly dependent, 
we may assume, without loss of generality, that $\lambda_j$ is rational 
for every $j\in \N$. 
Up to removing the eigenvalues equal to zero, if they exist, and changing the sign of some eigenvalues 
we can also assume that 
$$\lambda_j=\frac{a_j}{b_j},\quad \gcd(a_j,b_j)=1,\quad a_j, b_j>0$$
for every $j\in \N$. The boundedness and the infinite cardinality of the spectrum of $A$ imply that the sequence $(b_j)_{j\in\N}$ is unbounded.

In order to prove Lemma~\ref{lem:1234}, let us make some preliminary considerations. 
For every $m\in \N$ and $k_1,\dots,k_m\in \N$, 
denote by $\tau(k_1,\dots,k_m)$ the 
minimum of all $t>0$ such that $t\lambda_{k_j}$ belongs to $\N$ for every $j=1,\dots,m$. 
Equivalently said, $2\pi\tau({k_1},\dots,{k_m})$ is the period of the curve $s\mapsto e^{s A}\phi_{k_1}+\cdots+ e^{s A}\phi_{ k_m}$.
Notice that, if $\lambda_{k_1},\dots,\lambda_{k_m}$ are integers, then $\tau(k_1,\dots,k_m)=1/{\gcd(a_{k_{1}},\dots,a_{k_{m}})}$.
In general, 
\be\label{taum}
\tau({k_1},\dots,{k_m})=\frac{b_{k_1}\cdots b_{k_m}}{\gcd\lp a_{k_l}\Pi_{j=1}^{l-1}b_{k_j} \Pi_{j=l+1}^{m}b_{k_j}\rp_{1\leq l\leq m}}.
\ee

The following lemma guarantees that, for any choice of $k_1,\dots,k_{m-1}$, we can select $k_m$ in such a way that $\tau({k_1},\dots,{k_m})\gg \tau({k_1},\dots,{k_{m-1}})$. 

\begin{lemma}\label{5.2}
Let $m\in \N$.
For every $k_1,\dots,k_{m}\in \N$, there exists 
$c=c(k_1,\dots,k_{m})>0$ such that, for every $k_{m+1}\in\N$, 
$$\frac{\tau({k_1},\dots,{k_{m+1}})}{\tau({k_1},\dots,{k_{m}})}\geq \frac{b_{k_{m+1}}}{c}.$$ 
\end{lemma}

\begin{proof}
From equation~\r{taum},
$$
\frac{\tau({k_1},\dots,{k_{m+1}})}{\tau({k_1},\dots,{k_{m}})}=b_{k_{m+1}}\frac{\gcd\lp a_{k_l}\Pi_{j=1}^{l-1}b_{k_j} \Pi_{j=l+1}^{m}b_{k_j}\rp_{1\leq l\leq m}}{\gcd\lp a_{k_l}\Pi_{j=1}^{l-1}b_{k_j} \Pi_{j=l+1}^{{m+1}}b_{k_j}\rp_{1\leq l\leq {m+1}}}.
 $$
 
 The proof consists, then, in showing that $\gcd\lp a_{k_l}\Pi_{j=1}^{l-1}b_{k_j} \Pi_{j=l+1}^{{m+1}}b_{k_j}\rp_{1\leq l\leq {m+1}}$ is bounded from above by a constant
 independent of $k_{m+1}$. 
 First, notice that 
  $$
  \gcd\lp a_{k_l}\Pi_{j=1}^{l-1}b_{k_j} \Pi_{j=l+1}^{{m+1}}b_{k_j}\rp_{1\leq l\leq {m+1}}\leq \gcd\lp a_{k_{1}}\Pi_{j=2}^{{m+1}}b_{k_j},a_{k_{{m+1}}}\Pi_{j=1}^{m}b_{k_j}\rp=:\Gamma. 
  $$
 Set $c_1=  a_{k_1}\Pi_{j=2}^{m}b_{k_j}$ and $c_2= \Pi_{j=1}^{m}b_{k_j}$ and notice that they do  not depend on $k_{m+1}$.
 
 Write $\Gamma$ as
 $$\gamma_1\beta_{m+1}=\Gamma=\gamma_2 \alpha_{m+1}$$
  where $\gamma_1$ and $\gamma_2$ divide $c_1$ and $c_2$, respectively, while
  $\alpha_{m+1}$ and $\beta_{m+1}$ divide $a_{k_{m+1}}$ and $b_{k_{m+1}}$, respectively. 
  Since $a_{k_{m+1}}$ and $b_{k_{m+1}}$ are relatively prime, then the same is true for 
  $\alpha_{m+1}$ and $\beta_{m+1}$. Therefore, $\alpha_{m+1}$ divides $\gamma_1$. Hence, 
  $\Gamma=\alpha_{m+1} \gamma_2\leq \gamma_1 \gamma_2\leq c_1 c_2$.
\end{proof}

We show now how to choose $\bar k_1,\dots,\bar k_\ns$ as in the statement of Lemma~\ref{lem:1234}.
We proceed by induction on $\ns$. The case $\ns=1$ has already been treated in Section~\ref{phases-1}. Assume that $\bar k_1,\dots,\bar k_{\ns-1}$ are such that 
${\cal C}({\bar k_1},\ldots,{\bar k_{\ns-1}})$ is $\eta$-dense in ${\cal T}({\bar k_1},\ldots,{\bar k_{\ns-1}})$.
Hence, for every choice of $\bar k_\ns$,  the set
${\cal C}({\bar k_1},\ldots,{\bar k_{\ns-1}})+{\cal T}({\bar k_{\ns}})$ is $\eta$-dense in ${\cal T}({\bar k_1},\ldots,{\bar k_\ns})$. 

We are left to show that, for a suitable choice of $\bar k_\ns\in \N\setminus\{\bar k_1,\dots, \bar k_{\ns-1}\}$, the set   
${\cal C}({\bar k_1},\ldots,{\bar k_{\ns}})$ is $\eta'$-dense in ${\cal C}({\bar k_1},\ldots,{\bar k_{\ns-1}})+{\cal T}({\bar k_{\ns}})$ for $\eta'$ arbitrarily small.
Recall that ${\cal C}({\bar k_1},\ldots,{\bar k_{m}})$ is the support of the curve 
$s\mapsto e^{s A}\phi_{\bar k_1}+\cdots+ e^{s A}\phi_{\bar k_{m}}$, whose period equals $2\pi\tau(\bar k_1,\dots,\bar k_{m})$. Therefore, ${\cal C}({\bar k_1},\ldots,{\bar k_{\ns-1}})
$ is the projection of 
${\cal C}({\bar k_1},\ldots,{\bar k_{\ns}})$ 
along $\phi_{\bar k_\ns}$ on 
$\phi_{\bar k_\ns}^\perp$. 
Hence, for every 
$\psi\in {\cal C}({\bar k_1},\ldots,{\bar k_{\ns-1}})$ the 
cardinality of the set 
$(\psi+{\cal T}({\bar k_{\ns}}))\cap {\cal C}({\bar k_1},\ldots,{\bar k_{\ns}})$
is  equal to
$\tau({\bar k_1},\ldots,{\bar k_{\ns}})/\tau({\bar k_1},\ldots,{\bar k_{\ns-1}})$.
In particular, $(\psi+{\cal T}({\bar k_{\ns}}))\cap {\cal C}({\bar k_1},\ldots,{\bar k_{\ns}})$,  
which is regularly distributed, 
 is $2\pi \tau({\bar k_1},\ldots,{\bar k_{\ns-1}})/\tau({\bar k_1},\ldots,{\bar k_{\ns}})$-dense  in $\psi+{\cal T}({\bar k_{\ns}})$.

Lemma~\ref{5.2} and the unboundedness of the sequence $(b_j)_{j\in\N}$ allow to conclude.

\subsubsection{The spectrum of $A$ is bounded and $1<\mathrm{dim}_\Q\left(\mathrm{span}_\Q(\lambda_k)_{k\in\N}\right)<\infty$}

Let $m=\mathrm{dim}_\Q\left(\mathrm{span}_\Q(\lambda_k)_{k\in\N}\right)$ and fix a $\Q$-basis $\mu_1,\dots,\mu_m$ of $\mathrm{span}_\Q(\lambda_k)_{k\in\N}$.

Let 
$$\lambda_j=\sum_{l=1}^m\al_j^l\mu_l,\quad 
\al_j^1,\dots,\al_j^m\in\Q.$$

There exists $l\in\{1,\dots,m\}$ such that the cardinality of $\{\al_j^l\mid j\in\N\}$ is infinite. (Otherwise, the set of eigenvalues of $A$ would be finite.)
Without loss of generality, $l=1$. 

The results of Sections~\ref{sss-unbounded} and \ref{sss-oneQ}
imply that, for every $\eta>0$,  there exist $\bar k_1,\dots,\bar k_\ns\in\N$ such that 
$\{(e^{i t\mu_1 \al_{\bar k_1}^1},\dots,e^{i t\mu_1 \al_{\bar k_\ns}^1})\mid t\in\R\}$ is  $\eta$-dense in $\mathbb{T}^\ns$. 
We are going to show that ${\cal C}(\bar k_1,\dots,\bar k_\ns)$ is $\ns m\eta$-dense in ${\cal T}(\bar k_1,\dots,\bar k_\ns)$ or, equivalently, that
$\{(e^{i t\lambda_{\bar k_1}},\dots,e^{i t\lambda_{\bar k_\ns}})\mid t\in\R\}$ is  $\ns m\eta$-dense in $\mathbb{T}^\ns$.

Up to a reparametrization, we can assume that $\al_{\bar k_j}^l\in\Z$ for every $j=1,\dots,\ns$ and $l=1,\dots,m$. 

Fix $(e^{i\theta_1},\dots,e^{i \theta_\ns})$ in $\mathbb{T}^\ns$. 
The choice of $\bar k_1,\dots,\bar k_\ns\in\N$ guarantees the existence of $\bar t\in\R$ such that 
$\|(e^{i \bar t\mu_1 \al_{\bar k_1}^1},\dots,e^{i \bar t\mu_1 \al_{\bar k_\ns}^1})-(e^{i\theta_1},\dots,e^{i \theta_\ns})\|<\eta$. 
Because of the $\Q$-linear independence of $\mu_1,\dots,\mu_m$, there exists $t\in\R$ such that 
$$\| (e^{it\mu_1 },e^{it\mu_2 },\dots,e^{it\mu_m })-(e^{i\bar t\mu_1 },1,\dots,1)\|<\frac{\eta}{\max\{|\al_{\bar k_j}^l|\mid j=1,\dots,\ns,\;l=1,\dots,m\}}.$$

In particular $|e^{i  t\mu_l  \al_{\bar k_j}^l}-1 |<\eta$ for every $l=2,\dots,m$ and every $j=1,\dots,\ns$. 

Hence
\begin{align*}
\|(e^{i t\lambda_{\bar k_1}},\dots,e^{i t\lambda_{\bar k_\ns}})-(e^{i\theta_1},\dots,e^{i \theta_\ns})\|&\leq \|(e^{i t\lambda_{\bar k_1}},\dots,e^{i t\lambda_{\bar k_\ns}})-(e^{i \bar t\mu_1 \al_{\bar k_1}^1},\dots,e^{i \bar t\mu_1 \al_{\bar k_\ns}^1})\|\\
&\ \ +\|(e^{i \bar t\mu_1 \al_{\bar k_1}^1},\dots,e^{i \bar t\mu_1 \al_{\bar k_\ns}^1})-(e^{i\theta_1},\dots,e^{i \theta_\ns})\|\\
&\leq \sum_{j=1}^\ns |e^{i \bar t\mu_1 \al_{\bar k_j}^1}\cdots e^{i \bar t\mu_m \al_{\bar k_j}^m}- e^{i \bar t\mu_1 \al_{\bar k_j}^1}|+\eta\\
&\leq (\ns(m-1)+1)\eta.
\end{align*}

This concludes the proof of Lemma~\ref{lem:1234} and of Theorem~\ref{THE_Control_wave_function}.\hfill\EOP


\section{Example: Infinite potential well}\label{sec:Example}
We consider now the case of a particle confined in $(-1/2,1/2)$. This model has been extensively studied 
by several authors in the last few years and was the first quantum system for which a positive controllability result has been obtained. Beauchard proved exact controllability in some dense subsets of $L^2$ using Coron's return method (see \cite{beauchard-coron, beauchard-mirrahimi} for a precise statement). Nersesyan obtained approximate controllability results using Lyapunov techniques. In the following, we extend these controllability results to simultaneous controllability and provide some estimates of the $L^1$ norm of controls achieving the transfer between two density matrices.

The Schr\"{o}dinger equation writes
\begin{equation}\label{EQ_potential_well}
 i \frac{\partial \psi}{\partial t}=-\frac{1}{2} \frac{\partial^2 \psi}{\partial x^2} - u(t) x \psi(x,t)
\end{equation}
with the boundary conditions $\psi(-1/2,t)=\psi(1/2,t)=0$ for every $t \in \mathbf{R}$.

In this case ${\cal H}=L^2 \left ( (-1/2,1/2), \mathbf{C} \right )$ endowed with the Hermitian product $\langle \psi_1,\psi_2\rangle= \int_{-1/2}^{1/2} \overline{\psi_1(x)} \psi_2(x) dx$. The operators $A$ and $B$ are defined by $A\psi= i\frac{1}{2} \frac{\partial^2 \psi}{\partial x^2}$ for every $\psi$ in 
$D(A)= (H_2 \cap  H_0^1 )\left((-1/2,1/2), \mathbf{C}\right)$, and $B\psi=i x \psi$.        

Unfortunately, due to the numerous resonances,  we are not able to apply directly our results to system~(\ref{EQ_potential_well}). A classical approach is in this case to use perturbation theory. Indeed, consider for every $\eta$ in $[0,\delta]$, the operator $A_{\eta}=A+\eta B$. The controllability of system~(\ref{EQ_potential_well}) with control in $[0,\delta]$ is equivalent to the controllability of $$\frac{d\psi}{dt}= A_{\eta} \psi + v B \psi$$ with controls $v$ taking values in $[-\eta,\delta-\eta]$. 
Since  the perturbation $\eta\mapsto A_\eta$ is analytic, 
the self-adjoint operator $A_\eta$ admits a complete set of eigenvectors $(\phi_k(\eta))_{k \in \mathbf{N}}$  associated with the eigenvalues $(i\lambda_k(\eta))_{k \in \mathbf{N}}$, with $\phi_k$ and $\lambda_k$ analytic (see~\cite{katino}).
For $\eta=0$, 
$$ \phi_k(0)=
				\left \{ \begin{array}{ll}
                                  x \mapsto \sqrt{2} \cos(k\pi x)& \mbox{ when } k \mbox{ is odd}\\
 				  x \mapsto \sqrt{2} \sin(k\pi x)& \mbox{ when } k \mbox{ is even}
                                       \end{array}
				\right.$$
is a complete set of eigenvectors of $A$ associated with the eigenvalues $i\lambda_k(0)=  - i\frac{k^2 \pi^2}{2}$.
Following~\cite[Proposition 2.3]{beauchard-mirrahimi}, one can compute the 2-jet at zero of the analytic functions $\lambda_k$:
$$
\lambda_k(\eta)=  - \frac{k^2 \pi^2}{2} - \left ( \frac{1}{24 \pi^2 k^2}- \frac{5}{8 \pi^4 k^4} \right ) \eta^2 +\mathrm{o}(\eta^2),
$$
as $\eta$ goes to zero.
For every $k_1,k_2,p_1,p_2$ in $\mathbf{N}$, since $\frac{1}{\pi^2}$ is transcendental on $\mathbf{Q}$, 
$$\begin{array}{l}
\lambda_{k_1}''(0)-\lambda_{k_2}''(0)=\lambda_{p_1}''(0)-\lambda_{p_2}''(0)
\end{array}  \Leftrightarrow
\left \{\begin{array}{l}
       \frac{1}{k_1^2}-\frac{1}{k_2^2}=\frac{1}{p_1^2}-\frac{1}{p_2^2}\\
       \frac{1}{k_1^4}-\frac{1}{k_2^4}=\frac{1}{p_1^4}-\frac{1}{p_2^4}
        \end{array}
\right.
$$
$$
\Leftrightarrow
\left \{\begin{array}{l}
        \frac{1}{k_1^2}-\frac{1}{k_2^2}=\frac{1}{p_1^2}-\frac{1}{p_2^2}\\
        \left (\frac{1}{k_1^2}-\frac{1}{k_2^2} \right )\left (\frac{1}{k_1^2}+\frac{1}{k_2^2} \right )
= \left (\frac{1}{p_1^2}-\frac{1}{p_2^2} \right )\left (\frac{1}{p_1^2}+\frac{1}{p_2^2} \right )
        \end{array}
\right.
\Leftrightarrow
\left \{\begin{array}{l}
         k_1=p_1\\
	 k_2=p_2.
        \end{array}
\right.
$$
Hence, the 2-jets of the gaps between two eigenvalues are all different at zero. Recall that 
$\langle \phi_k, B \phi_{k+1} \rangle \neq 0$. An argument similar to the one in \cite[Proposition 6.2]{Schrod} ensures that for every $\eps >0$
 there exists $0 < \eta < \eps $ such that $\{(j,j + 1), (j+1,j)\, |\, j \in \N \}$ is a non-resonant connectedness chain for $(A_{\eta},B,[-\eta,\delta - \eta], \Phi)$. 

We can now  apply Theorem~\ref{THE_Control_wave_function} to obtain simultaneous approximate controllability for (\ref{EQ_potential_well}) and 
Theorem~\ref{PRO_L1_estimates} to get $L^{1}$
estimates. For instance, for every $\eps > 0$ 
there exist $T>0$ and  a piecewise constant function $u:[0,T] \rightarrow [0,\delta]$ such that the propagator at time $T$ of (\ref{EQ_potential_well}) exchanges (up to a correction of size $\eps$) the density matrices 
$$
\frac{1}{3} \phi_1(0) \phi_1(0)^{\ast} + \frac{2}{3} \phi_2(0) \phi_2(0)^{\ast} \quad \mbox{ and } \quad 
\frac{1}{3} \phi_2(0) \phi_2(0)^{\ast} + \frac{2}{3} \phi_1(0) \phi_1(0)^{\ast}
$$
 and 
$$
\|u\|_{L^1} \leq \frac{\pi }{2 \nu |\langle \phi_1(0), B \phi_2(0) \rangle|}=\frac{9 \pi^3}{ 32 \nu }\approx 20.2656\,.
$$
The control time $T$ satisifies $T\geq \frac{1}{\delta}\|u\|_{L^1}$ and, from Proposition \ref{PRO_L1_lower_estimates}, the control $u$ has to satisfy
$$
\|u\|_{L^1}\geq (1-\eps) \max \left \{ \frac{2\,\sqrt{3}\,\pi}{\sqrt{{\pi}^{2}-6}},
\frac{2\,\sqrt{6}\,\pi}{\sqrt{2\,{\pi}^{2}-3}} \right \}\approx (1-\eps) 5.5323\,.
$$

\section{Orientation of a bipolar molecule in the plane by means of two external fields}\label{sec:molecule}

We present here 
an
example of approximately controllable quantum system. (See~\cite{seideman,spanner,stapelfeldt} and references therein.) 
It provides a simple model for the control by two electric fields of the rotation of a bipolar rigid molecule confined to a plane (see Figure~\ref{fig:molecola}). Molecular orientation and alignment are well-established topics in the quantum control of molecular dynamics both from the experimental and theoretical point of view. 

The model we aim to consider can be represented  
by a Schr\"odinger equation on the circle $\mathbb{S}^1=\R/2\pi\Z$, so that 
${\CH}=L^2(\mathbb{S}^1,\C)$.
In this case $A=i\;{\partial^2}/{\partial \theta^2}$ has discrete spectrum and its eigenvectors are trigonometric functions.
The controlled Schr\"odinger equation is 
\begin{equation}\label{Sch-rot}
i\frac{\partial\psi(\theta,t)}{\partial t}=\left(  -\frac{\partial^2}{\partial \theta^2}+u_1(t) \cos(\theta)+u_2(t) \sin(\theta) \right)\psi(\theta,t).
\end{equation}
We assume that both $u_1$ and $u_2$ are piecewise constant
and take values in $[0,\delta]$, $\delta >0$.

Notice that system~\eqref{Sch-rot} is not controllable if we fix one control to zero. 
Indeed, by parity reasons,
the potential $\cos(\theta)$ does not couple an odd wave function with an even one and the potential
 $\sin(\theta)$ does not
 couple  wave functions with the same parity.


For every $\alpha\in \mathbb{S}^1$, let us split $\CH$ as $\CH^\alpha_e\oplus\CH^\alpha_o$, where $\CH^\alpha_e$ (respectively, $\CH^\alpha_o$) is the closed subspace of $\CH$ of even (respectively, odd)  functions 
with respect to $\alpha$. Notice that
 $\CH^\alpha_e$ and $\CH^\alpha_o$ are Hilbert spaces.
 A complete orthonormal system for $\CH^{\alpha}_{e}$ (respectively, $\CH^\alpha_o$) is given by $\{\cos(k(\cdot - \alpha))/\sqrt{\pi}\}_{k=0}^{\infty}$
(respectively,  $\{\sin(k(\cdot - \alpha))/\sqrt{\pi}\}_{k=1}^{\infty}$).
 Let us 
write $\psi=\psi^\alpha_e+\psi^\alpha_o$, where $\psi^\alpha_e\in\CH^\alpha_e$ and  $\psi^\alpha_o\in\CH^\alpha_o$.

Our first result states that, once $\alpha$ is fixed, the even or the odd part 
of a wave function $\psi$ can be approximately controlled, under the constraint  that 
their $L^2$-norms are preserved.  
\begin{lem}\label{th-fixed}
Let $\psi\in \CH$ and $\alpha\in [0,\pi/2]$. Then 
for  every $\eps>0$ and every $\tilde\psi\in \CH_e^\alpha$ such that $\|\psi_e^\alpha\|_\CH=\|\tilde\psi\|_\CH$, there exists a piecewise constant control $u=(u_1,u_2)$ steering $\psi$  to a wave function whose 
even part with respect to $\alpha$ lies in an $\eps$-neighborhood of $\tilde \psi$.
Similarly, 
for  every $\eps>0$ and every $\tilde\psi\in \CH_o^\alpha$ such that $\|\psi_o^\alpha\|_\CH=\|\tilde\psi\|_\CH$, there exists a piecewise constant control $u
$ steering $\psi$  to a wave function whose 
odd part with respect to $\alpha$ lies in an $\eps$-neighborhood of $\tilde \psi$.
\end{lem}
\begin{proof}
The idea is to apply Theorem~\ref{THE_Control_collectively} 
to a subsystem of \eqref{Sch-rot} corresponding to the choice of a 
subclass of wave functions and a subclass of admissible controls. 

More precisely, let 
$${\cal U}_\alpha=\{u: \R \to [0,\delta]^2 \mbox{ piecewise constant }\mid u(t) 
\mbox{ is proportional to }
(\cos\alpha,\sin\alpha)\mbox { for all $t\in\R$}\}.$$
The control system whose dynamics are described by \eqref{Sch-rot}  with admissible control functions restricted to ${\cal U}_\alpha$ can be rewritten as
\begin{equation}\label{Sch-rot-theta}
i\frac{\partial\psi(\theta,t)}{\partial t}=\left(  -\frac{\partial^2}{\partial \theta^2}+v(t) \cos(\theta-\alpha) \right)\psi(\theta,t),\quad v\in
\left(0,\delta\sqrt{1+\min\{\tan\alpha,\mathrm{cotan}\,\alpha\}^2}\right).
\end{equation}

Notice that the spaces $\CH_e^\alpha$ and $\CH_o^\alpha$ are invariant for the evolution of \eqref{Sch-rot-theta}, whatever the choice of $v=v(\cdot)$, since 
they are invariant 
both for 
$A=i \frac{\partial^2}{\partial \theta^2}$ 
 and for the multiplicative operator $B^\alpha:\CH\to \CH$ defined by 
 $(B^{\alpha}\phi)(\theta)=-i\cos(\theta-\alpha)\phi(\theta)$. 

We shall consider \eqref{Sch-rot-theta} as a control system 
defined on $\CH_e^\alpha$ (the second part of the statement of the lemma can be proved similarly by considering its dynamics restricted to $\CH_o^\alpha$). 

Denote by $A_e^\alpha$ and $B_e^\alpha$ the restrictions of $A$ and $B^\alpha$ to $\CH_e^\alpha$. 
Choose as orthonormal basis of eigenfunctions for $A_e^\alpha$ 
the sequence 
defined by $\phi_k(\theta)=\cos(k(\theta-\alpha))/\sqrt{\pi}$ for $k\in \N$.
The eigenvalue of $A$ associated with $\phi_k$ is $i\lambda_k=-i k^2$. 

Then $\langle \phi_j,B_e^\alpha\phi_k\rangle \ne 0$ if and only if 
$|k-j|=1$. We take as connectedness chain the set $\{(k,j)\in \N^2\mid |k-j|=1\}$.
Since $\lambda_{k+1}-\lambda_k=-2k-1$, then the connectedness chain is non-resonant. 
Theorem~\ref{THE_Control_collectively} implies that 
\eqref{Sch-rot-theta} can be steered from 
$\psi_e^\alpha$ to an $\eps$-neighborhood of any $\tilde\psi\in \CH_e^\alpha$ such that 
$\|\psi_e^\alpha\|=\|\tilde\psi\|$
by an admissible control
$v(\cdot)$. The conclusion follows by applying $u(\cdot)=(v(\cdot)\cos\alpha,v(\cdot)\sin(\alpha))$ to 
\eqref{Sch-rot} (since $\CH_e^\alpha$ and $\CH_o^\alpha$ are invariant by the flow generated by $u(\cdot)$). 
\end{proof}


The main result of this section states that 
\eqref{Sch-rot} is approximately controllable. 
\begin{prop}\label{rot:controllability}
System \eqref{Sch-rot} is approximately controllable. 
\end{prop}
\begin{proof}
%
%
%
%
It is enough to prove that every wave function of norm one can be steered arbitrarily close to the constant
$1/\sqrt{2\pi}$. Indeed, 
 if $\psi$ and $\tilde \psi$ have norm one, 
 if the control $u(\cdot)$ steers the initial condition $\psi$ $\eps$-close to the constant $1/\sqrt{2\pi}$, and if $\tilde u(\cdot)$ steers the conjugate of $\tilde \psi$ $\eps$-close to the same constant $1/\sqrt{2\pi}$, then the concatenation of $u$ and of the time reversed of  $\tilde u$  steers $\psi$ $2\eps$-close to $\tilde \psi$.

Fix $\psi\in \CH$ of norm one, a tolerance $\eps>0$, and choose $\alpha\in (0,\pi/2)$. 
Fix $\bar{\eps}=\eps(1 - 1/\sqrt{2})$.  Then, according to Lemma~\ref{th-fixed}, $\psi$ can be steered to a wave function $\tilde\psi$ 
such that $\|\tilde\psi-\psi_{1} \|<\bar{\eps}$, 
where
$\psi_{1}$ is of the form
$$
\psi_{1} = \frac{\|\psi_e^\alpha\|}{\sqrt{2\pi}} + \phi_1, \quad \mathrm{with}\ \phi_1\in\CH_o^\alpha.
$$
If $\|\phi_1\|$ is smaller than $\eps/2$ then 
we are done. Indeed, $\phi_1$ has $L^2$-norm equal to $\sqrt{1-\|\psi_e^\alpha\|^2}$, which implies that $1-\|\psi_e^\alpha\| < \eps^{2}/4$. Then
\begin{align*}
\left\|\tilde\psi - \frac{1}{\sqrt{2\pi}}\right\| &\leq \|\tilde\psi - \psi_{1}\| + \left\|\psi_{1} - \frac{1}{\sqrt{2\pi}}\right\|< \bar{\eps} + \sqrt{(\|\psi_e^\alpha\|-1)^{2} + \|\phi_1\|^{2}}\\
& = \bar{\eps} + \sqrt{2}\sqrt{1 - \|\psi_e^\alpha\|} < \bar{\eps} + \frac{1}{\sqrt{2}}\eps = \eps.
\end{align*}

Assume then that $\|\phi_1\|\geq \eps/2$ and consider,
for every $\beta\in \mathbb{S}^1$, $\tau_\beta = \|(\phi_{1})^{
\beta}_e\|^{2}$. 
We can characterize $\tau_\beta$ in terms of the coefficients $a_k$ 
of the representation
$$\phi_1(\cdot)=\sum_{k=1}^\infty a_k \frac{\sin(k(\,\cdot -\alpha))}{\sqrt{\pi}}.$$
Indeed, 
\begin{align*}
\tau_\beta^{2}&=
\sum_{k=1}^\infty\left|\langle\phi_1(\cdot),\cos(k(\cdot-\beta))/\sqrt{\pi}\rangle\right|^2
= \sum_{k=1}^{\infty} \left|\langle a_{k}\sin(k(\,\cdot -\alpha))/\sqrt{\pi},\cos(k(\cdot-\beta))/\sqrt{\pi}\rangle\right|^2\\
&=
\sum_{k=1}^\infty\left|a_k\sin((\beta-\alpha)k)\right|^2.
\end{align*}
There exists $c>0$ independent of $k$ and $\alpha$ such that
$$\int_0^{\frac\pi2}\sin^2((\beta-\alpha)k)d\beta\geq c,$$
hence, 
$$\int_0^{\frac\pi2}\sum_{k=1}^\infty|a_k|^2\sin^2((\beta-\alpha)k)d\beta\geq c\sum_{k=1}^\infty|a_k|^2=c \|\phi_1\|^2,$$
from which we conclude that there exists $\beta\in (0,\pi/2)$ such that
$$
\tau_\beta\geq \frac {2c}{\pi}\|\phi_1\|^{2}\geq \frac {c}{2\pi}\eps^{2}.
$$

Notice now that the even part of 
$\psi_{1}$
with respect to $\beta$ has norm 
$$\|(\psi_1)_e^\beta\|=\sqrt{\|\psi_e^\alpha\|^2+\tau_\beta}\geq \sqrt{\|\psi_e^\alpha\|^2+\frac{c\eps^2}{2\pi}}.$$ 
Repeating the same argument as above replacing $\psi$ by $\psi_{1}$
we conclude that it is then possible to steer 
$\psi$ 
at a distance smaller than
$\bar\eps$ 
from the sum $\psi_{2}$ of a positive constant function of norm larger than $\sqrt{\|\psi_e^\alpha\|^2+\frac{c\eps^2}{2\pi}}$ and a function $\phi_{2}\in\CH_o^\beta$ of norm $\|\phi_{2}\| < \|\phi_{1}\|$. If $\|\phi_{2}\|<\eps/2$ then we are done. Otherwise,
since the improvement in the size of the constant is bounded from below by a quantity that does not depend on $\phi_1$, we can iterate the procedure finitely many times up to guaranteeing that the final wave function is $\eps$-close to the constant  $1/\sqrt{2\pi}$. 
\end{proof}


\appendix
\section{Appendix: relations between controllability notions}


\begin{prop}\label{prop:relations}
Let $(A,B, U, \Phi)$ satisfy $(\mathfrak{A})$. 
Then
\begin{itemize}
\item[$(i)$] for every $\ns$ in $\mathbf{N}$, $\eps>0$, and $\uni$ in $\mathbf{U}(\cal H)$  there exists a piecewise constant control $u:[0,T]\rightarrow U$ such that
 such that
$
 \| \uni \phi_k - \pro^{u}_T \phi_k  \|<\eps
$
 for every $1\leq k \leq \ns$,
\end{itemize}
implies
\begin{itemize}
\item[$(ii)$] \eqref{eq:main} is controllable in the sense of densities matrices
\end{itemize}
 which implies
 \begin{itemize}
 \item[$(iii)$]
for every $\ns$ in $\mathbf{N}$, $\eps>0$, and $\uni$ in $\mathbf{U}(\cal H)$ there exist $\theta_1,\ldots,\theta_{\ns} \in \R$ and a piecewise constant control $u:[0,T]\rightarrow U$ such that $\|e^{i \theta_k}\uni(\phi_k) - \pro^{u}_T(\phi_k)\|<\eps$ for every $1\leq  k \leq \ns$.
\end{itemize}
\end{prop}

\begin{proof}
Let us prove that $(i)$ implies $(ii)$.
Fix  two unitarily equivalent density matrices $\rho_0$ and $\rho_1$. Write $\rho_0=\sum_{k \in \mathbf{N}} P_k v_k v^{\ast}_k$ with $(v_k)_{k\in \N}$ an orthonormal sequence in $\cal H$ and $(P_k)_{k\in \N}$ a sequence in $\ell^1([0,1])$ such that $\sum_k P_k =1$. By assumption, there exists $\uni$ in $\mathbf{U}({\cal H})$ such that $\rho_1=\uni \rho_0 \uni^{\ast}=\sum_{k\in\N} P_k \uni(v_k) \uni(v_k)^{\ast}$.

Let $m$ in $\mathbf{N}$ be such that $\sum_{k>m} P_{k} <\eps$
and  $\ns$ in $\N$ be
 such that $\|v_j - \sum_{k=1}^{\ns} \la \phi_{k}, v_{j} \ra \phi_{k}\|<\eps$ for every $j=1,\ldots,m$.
 By hypothesis, there exists a piecewise constant function $u:[0,T]\rightarrow U$ such that
$\|  \uni(\phi_k)  - \pro^{u}_T(\phi_k) \|<\eps/\ns$ for every $ 1\leq k \leq \ns$.
Hence, for $j=1,\ldots,m$,
\begin{align*}
\| \uni(v_j)  - \pro^{u}_T(v_j) \| & = \left\| \sum_{k=1}^{\infty}  \la\phi_{k},v_{j} \ra (\uni(\phi_k)  - \pro^{u}_T(\phi_k) ) \right\| \\
&\leq \sum_{k=1}^{\ns}\|  \la\phi_{k},v_{j} \ra (\uni(\phi_k)  - \pro^{u}_T(\phi_k))  \| + 
\left\| \uni\left(\sum_{k=\ns+1}^{\infty}  \la\phi_{k},v_{j} \ra \phi_{k}\right)\right\|   \\
&\quad+ \left\|  \pro^{u}_T\left(\sum_{k=\ns+1}^{\infty}  \la\phi_{k},v_{j} \ra \phi_{k}\right)\right\| \\ 
& \leq \eps + 2 \left\|v_j - \sum_{k=1}^{\ns} \la \phi_{k}, v_{j} \ra \phi_{k}\right\|  \leq 3 \eps.
\end{align*}
Then, recalling that for every $\theta$ in $\mathbf{R}$,  $a,b$ in $\cal H$, $\|a a^{\ast}-b b^{\ast}\|\leq \|aa^{\ast}-e^{i \theta} b a^{\ast} + e^{i \theta} b a^{\ast}- b b^{\ast}\| 
\leq \|a\| \|a-e^{i \theta}b\| + \|b \| \|e^{i \theta} a^{\ast}-b^{\ast}\| \leq 
 (\|a\| + \|b\|)\|a-e^{i \theta} b\|$, we get
\begin{eqnarray*}
 \| \pro^{u}_T \rho_0 \pro^{u^{\ast}}_T-\rho_1\| & \leq &
 \sum_{j =1}^{m} P_k \| \pro^{u}_T(v_j) \pro^{u}_T(v_j)^{\ast}-\uni(v_j)\uni(v_j)^{\ast}\| + 2 \eps\\
& \leq & \sum_{j =1}^{m} P_k (\|\pro^{u}_T(v_j)\| + \|\uni(v_j)\|)\|\pro^{u}_T(v_j)-\uni(v_j)\| + 2 \eps\\
&\leq& 6 \eps +2\eps=8\eps,
\end{eqnarray*}
which concludes the first part of the proof.

Assume now that $(ii)$ holds true.
Fix $\eps$, $\ns\in\N$,  and $\uni$ as in the hypotheses. Choose $a_{1},\ldots, a_{\ns} \in \R$  such that $0<a_1<a_2<\cdots<a_\ns$ and $\sum_{k=1}^\ns a_k=1$.  Define  
 the two unitarily equivalent density matrices
 $$
 \rho_0=\sum_{k=1}^\ns a_k \phi_k \phi_k^{\ast} \quad \mbox{ and } \quad \rho_1=\sum_{k=1}^\ns  a_k \uni(\phi_k) \uni(\phi_k)^{\ast}.
 $$
 By assumption, there exists a piecewise constant $u:[0,T]\rightarrow U$ such that 
\begin{equation}\label{EQ_ineq_density_matrices_criteron_contr}
 \| \rho_1 -\pro^{u}_T \rho_0 \pro^{u*}_T\|< C \eps,
\end{equation}
where $C=\min\{a_{j}, |a_{k}-a_{l}|\mid 1\leq j,k,l \leq \ns, k \neq l\}/2$.
Choose $1\leq k_0\leq \ns$ and test~\eqref{EQ_ineq_density_matrices_criteron_contr} 
on $\pro^{u}_T \phi_{k_0}$
$$
\left \| a_{k_0} \pro^{u}_T \phi_{k_0}- \sum_{k=1}^\ns a_k \langle \uni(\phi_k),\pro^{u}_T \phi_{k_0} \rangle \uni(\phi_k) \right \| < C \eps. 
$$
Since $(\uni(\phi_k))_k$ is an Hilbert basis of $\cal H$, then
$$\pro^{u}_T \phi_{k_0}= \sum_{k=1}^\infty \langle \uni(\phi_k),\pro^{u}_T \phi_{k_0} \rangle \uni(\phi_k),$$
and hence
$$
\left\|\sum_{k=1}^\ns (a_{k_0}-a_k) \langle \uni(\phi_k),\pro^{u}_T \phi_{k_0} \rangle \uni(\phi_k) + a_{k_0} 
\sum_{k>\ns} \langle \uni(\phi_k),\pro^{u}_T \phi_{k_0} \rangle \uni(\phi_k)\right\|
<C \eps. 
$$
In particular we have
$$
 \left\| \sum_{k\neq k_0} \langle \uni(\phi_k),\pro^{u}_T \phi_{k_0} \rangle \uni(\phi_k) \right\|< \frac{\eps}{2}.
 $$
For $\eps$  small enough, this leads to 
$
| \langle \uni(\phi_{k_{0}}),\pro^{u}_T \phi_{k_0} \rangle|>\sqrt{1-\eps^2/4}>1-\eps/2.
$
Hence,
there exists $\theta_{k_{0}}$ in $\mathbf{R}$ such that $\|e^{i \theta_{k_{0}}} \uni(\phi_{k_{0}})-\pro^{u}_T (\phi_{k_{0}})\| < \eps$. 
\end{proof}


\begin{prop}
Under the hypotheses of Proposition~\ref{prop:relations} and if, moreover, $\la \phi_{j}, B \phi_{k}\ra$ is purely imaginary for every $j,k \in \N$,  then
$(iii)$ implies $(ii)$.
\end{prop}

\begin{proof}
First of all, because of the hypotheses on $B$, system~\eqref{eq:main} satisfies the following time-reversibility property:
If $\psi: [0,T] \to \CH$ is a solution of~\eqref{eq:main} with control $u(\cdot)$, then  
$\varphi(t) = 
\sum_{k=1}^{\infty}  \overline{\la  \phi_{k}, \psi(T-t) \ra} \phi_{k} $ is a solution of~\eqref{eq:main} with control $u(T-\cdot)$.

Fix  two unitarily equivalent density matrices $\rho_0$ and $\rho_1$. Write $\rho_0=\sum_{k \in \mathbf{N}} P_k v_k v^{\ast}_k$ with $(v_k)_{k\in \N}$ an orthonormal sequence in $\cal H$ and $(P_k)_{k\in \N}$ a sequence in $\ell^1([0,1])$ such that $\sum_{k\in\N} P_k =1$. Let $\uni$ in $\mathbf{U}({\cal H})$ be such that $\rho_1=\uni \rho_0 \uni^{\ast}=\sum_{k\in\N} P_k \uni(v_k) \uni(v_k)^{\ast}$.

Let $\ns\in \N$ be such that $\sum_{k > \ns} P_{k} < \eps$. By hypothesis there exist a control 
$u:[0,T_{u}] \to [0,\delta]$ and $\theta_1,\ldots,\theta_{\ns} \in \R$ such that 
$\|e^{i \theta_k} \sum_{j=1}^{\infty} \overline{\la \phi_{j},v_k \ra} \phi_{j} - \pro^{u}_{T_{u}}(\phi_k)\|<\eps$ for every $1\leq  k \leq \ns$. By time reversibility the control $u (T_{u}-\cdot)$ 
steers $e^{- i \theta_k}v_k $ $\eps$-close to $ \phi_k$.
Now, there exist a control $w:[0,T_w]\to[0,\delta]$ and $\hat{\theta}_1, \ldots, \hat{\theta}_\ns$  such that
$$
\| e^{i \hat{\theta}_k} \uni(v_k) - \pro^{w}_{T_{w}}(\phi_k)  \|\,.
$$
Let $\tilde u$  be the concatenation of the controls $u(T_u-\cdot)$ and $w$.
 Then
$$
\| e^{i \hat{\theta}_k} \uni(v_k) - \pro^{\tilde{u}}_{T_{\tilde{u}}}(e^{-i {\theta}_k }v_k)  \| < 2 \eps \,.
$$
Finally,
\begin{align*}
\|\pro^{\tilde{u}}_{T_{\tilde{u}}} \rho_{0} \pro^{\tilde{u}^\ast}_{T_{\tilde{u}}} - \rho_{1}\| &\leq
\sum_{k=1}^{\ns} P_{k}\|\pro^{\tilde{u}}_{T_{\tilde{u}}}(e^{-i {\theta}_k }v_k)\pro^{\tilde{u}}_{T_{\tilde{u}}}(e^{-i {\theta}_k }v_k)^{\ast} - 
\uni(v_{k})
\uni(v_{k})
^{\ast}\| + \\
& \qquad + \sum_{k=\ns+1}^{\infty} P_{k}\|\pro^{\tilde{u}}_{T_{\tilde{u}}}(v_k)\pro^{\tilde{u}}_{T_{\tilde{u}}}(v_k)^{\ast} - 
\uni(v_{k})\uni(v_{k})^{\ast}\| \\
& \leq   4\eps \sum_{k=1}^{\ns} P_{k}  + 2\eps \leq 6 \eps.
\end{align*}
This concludes the proof.
\end{proof}


\section*{Acknowledgements}
This research has been supported by the European Research Council, ERC
StG 2009 ``GeCoMethods'', contract number 239748, by the ANR project
GCM, program ``Blanche'',
project number NT09-504490, and by the Inria Nancy-Grand Est ``COLOR'' project.
 
The authors wish to thank the {\it Institut Henri Poincar\'e} (Paris, France) for providing research facilities and a stimulating environment during the ``Control of Partial and Differential Equations and Applications'' program in the Fall 2010.

{\small
\bibliographystyle{abbrv}
\bibliography{biblio}
}

\end{document}